\documentclass[12pt]{article}
\usepackage{amssymb}
\usepackage{amsmath}
\usepackage{graphicx}

\begin{document}
\newtheorem{theorem}{Theorem} 
\newtheorem{tha}{Theorem}
\newtheorem{conjecture}[theorem]{Conjecture}
\renewcommand{\thetha}{\Alph{tha}}

\newtheorem{corollary}[theorem]{Corollary}
\newtheorem{lemma}[theorem]{Lemma}
\newtheorem{proposition}[theorem]{Proposition}
\newtheorem{construction}[theorem]{Construction}
\newtheorem{claim}[theorem]{Claim}
\newtheorem{definition}[theorem]{Definition}
\newtheorem{observation}{Observation}
\newtheorem{question}{Question}
\newtheorem{remark}[theorem]{Remark}
\newtheorem{algorithm}[theorem]{Algorithm}
\newtheorem{example}[theorem]{Example}

\def\endproofbox{\hskip 1.3em\hfill\rule{6pt}{6pt}}
\newenvironment{proof}%
{%
\noindent{\it Proof.}
}%
{%
 \quad\hfill\endproofbox\vspace*{2ex}
}
\def\qed{\hskip 1.3em\hfill\rule{6pt}{6pt}}

\parindent=15pt

\def\dim{{\rm dim}}
\def \e{\epsilon}
\def\vec#1{{\bf #1}}
\def\va{\vec{a}}
\def\vb{\vec{b}}
\def\vc{\vec{c}}
\def\vx{\vec{x}}
\def\vy{\vec{y}}
\def\cM{{\mathcal M}}
\def\cH{{\cal H}}
\def\cF{{\mathcal F}}
\def\bE{\mathbb{E}}
\def\bP{\mathbb{P}}
\def\cA{{\mathcal A}}
\def\cB{{\mathcal B}}
\def\cJ{{\mathcal J}}
\def\cM{{\mathcal M}}
\def\cS{{\mathcal S}}
\def\cT{{\mathcal T}}
\def\wt#1{\widetilde{#1}}
\title{New Lower Bounds for Permutation Arrays Using Contraction}
\author{ Sergey Bereg \thanks{ Computer Science Dept., University of Texas at Dallas,
Richardson, TX 75083, USA} 
\and Zevi Miller \thanks{Dept. of Mathematics, Miami University, Oxford, OH 45056, USA}
\and Luis Gerardo Mojica $^*$ 
\and Linda Morales $^*$
\and I.H. Sudborough$^*$ }

\date{August 28, 2018}
\maketitle

\begin{abstract}
A {\it permutation array} $A$ is a set of permutations on a finite set $\Omega$, say of size $n$.  
Given distinct permutations $\pi, \sigma\in \Omega$, 
we let $hd(\pi, \sigma) = |\{ x\in \Omega: \pi(x) \ne \sigma(x) \}|$, called the {\it Hamming distance} between $\pi$ and $\sigma$.  
Now let $hd(A) =$ min$\{ hd(\pi, \sigma): \pi, \sigma \in A \}$.  For positive integers $n$ and $d$ with $d\le n$, we 
let $M(n,d)$ be the maximum number of permutations in any array $A$ satisfying $hd(A) \geq d$.  There is an extensive literature on 
the function $M(n,d)$, motivated in part by suggested applications to error correcting codes for message transmission over power lines.  

A basic fact is that if a permutation group $G$ is sharply $k$-transitive on a set of size $n\geq k$, then $M(n,n-k+1) = |G|$.  
Motivated by this we consider the permutation groups 
$AGL(1,q)$ and $PGL(2,q)$ acting sharply $2$-transitively on $GF(q)$ 
and sharply $3$-transitively on $GF(q)\cup \{\infty\}$ respectively.  Applying a contraction operation 
to these groups, we obtain the following new lower bounds for prime powers $q$ satisfying $q\equiv 1$ (mod $3$).

\begin{enumerate}
\item
$M(q-1,q-3)\geq (q^{2} - 1)/2$ for $q$ odd, $q\geq 7$,
\item
$M(q-1,q-3)\geq (q-1)(q+2)/3$ for $q$ even, $q\geq 8$,
\item
$M(q,q-3)\geq Kq^{2}\log q$ for some constant $K$ if $q$ is odd, $q\geq 13$.
\end{enumerate}

These results resolve a case left open in a previous paper \cite{BLS}, where it was shown that $M(q-1, q-3) \geq q^{2} - q$ and $M(q,q-3) \geq q^{3} - q$ 
for all prime powers $q$ such that $q\not \equiv 1$ (mod $3$).  We also obtain lower  bounds for $M(n,d)$ for a finite number of exceptional pairs $n,d$, by applying 
this contraction operation to the sharply $4$ and $5$-transitive Mathieu groups. 
\end{abstract} 

\section{Introduction}

\subsection{Notation and General Background}

We consider permutations on a set $\Omega$ of size $n$.  Given two such permutations $\pi$ and $\sigma$, 
we let $hd(\pi, \sigma) = |\{ x\in \Omega: \pi(x) \ne \sigma(x) \}|$, so $hd(\pi, \sigma)$ is the number 
of elements of $\Omega$ at which $\pi$ and $\sigma$ disagree.  When $hd(\pi, \sigma) = d$, we say that $\pi$ 
and $\sigma$ and are at {\it Hamming distance $d$}, or that the {\it Hamming distance between} $\pi$ and $\sigma$ is $d$.  A {\it permutation array} $A$ is a 
set of permutations on $\Omega$.  We say 
that $hd(A) = d$ if $d = min\{ hd(\pi, \sigma): \pi, \sigma \in A \}$.  For positive integers $n$ and $d$ with $d\le n$ we 
let $M(n,d)$ be the maximum number of permutations in any permutation array $A$ satisfying $hd(A) \geq d$.

Consider a fixed ordering $x_{1}, x_{2}, \cdots, x_{n}$ of the elements of $\Omega$.  The {\it image string} of the permutation $\sigma\in A$ is the string 
$\sigma(x_{1})\sigma(x_{2})\cdots\sigma(x_{n})$.  Thus the permutation array $A$ can also be regarded as an $|A|\times n$ matrix whose rows are 
the image strings of the permutations in $A$.  When $hd(A) = d$, any two rows of $A$ disagree in at least $d$ positions and some pair of rows disagree in 
exactly $d$ positions.  In particular, if $G$ is a permutation group acting on  $\Omega$, then we obtain a $|G|\times n$ permutation array whose rows 
consist of the mage strings of all the elements of $G$.  We refer to this array by $G$, and we use $hd(G)$ to refer to the hamming distance of this array.  

The study of permutation arrays began (to our knowledge) with the papers \cite{DeVa} and \cite{FrDe}, where good bounds on $M(n,d)$ (together with other results) were developed based on combinatorial 
methods, motivated by the Gilbert-Varshamov bounds for binary codes.  In recent years there has been renewed interest in permutation arrays, motivated by 
suggested applications in power line transmission \cite{FerVin}, \cite{PVYH}, \cite{Vin}, \cite{H}, 
block ciphers \cite{TorCoLi}, and in multilevel flash memories \cite{JMSB} and \cite{JSB}.

We review here some of the known results and methods for estimating $M(n,d)$.  

Some elementary exact values and bounds on $M(n,d)$ are the following (summarized with short proofs in \cite{CCD}) ; $M(n,2) = n!$, $M(n,3) = \frac{n!}{2}$, 
$M(n,n) = n$, $M(n,d)\geq M(n-1,d)$, $M(n,d)\geq M(n,d+1)$, $M(n,d)\le nM(n-1,d)$, and 
$M(n,d)\le \frac{n!}{(d-1)!}$.  More sophisticated bounds were developed in the above cited papers \cite{DeVa} and \cite{FrDe}, 
with a recent improvement in \cite{WZYG}.   
The smallest interesting case for $d$ is $d=4$.  Here some interesting and non-elementary 
bounds for $M(n,4)$ were developed in \cite{DS}, using linear programming, characters on the symmetric group $S_{n}$, and Young diagrams.  
In \cite{KK} it is shown that if $K>0$ is a constant with $n > e^{30/K^{2}}$ and $s < n^{1-K}$, then $M(n,n-s) \geq \theta(s! \sqrt{\log n})$.  The lower bound is 
achieved by a polynomial time randomized construction, using the Lovasz Local Lemma in the analysis.

There are various construction methods for permutation arrays.  First there is a connection with mutually orthogonal latin squares (MOLS).  It was shown in \cite{CKL} that 
if there are $m$ MOLS of order $n$, then $M(n,n-1) \geq mn$.  From this it follows that if $q$ is a prime power, then $M(q,q-1) = q(q-1)$.  Computational approaches for bounding 
$M(n,d)$ for small $n$ and $d$, including clique search, and the use of automorphisms are described in \cite{CCD}, \cite{JLOS}, and \cite{SM}.  There are also constructions of 
permutation arrays that arise from the use of permutation polynomials, also surveyed in \cite{CCD}, which we mention briefly below.

Additional construction methods are {\it coset search} \cite{BLS} and {\it partition and extension} \cite{BMS}.  In the first of these, one starts with with a permutation group $G$ on $n$ 
letters with $hd(G) = d$, 
and which is a subgroup of some group $H$ (for example $H = S_{n}$).  Now for disjoint permutation arrays $A,B$ on the same set of letters, let 
$hd(A,B) =$ min$\{hd(\sigma, \tau): \sigma\in A,\tau \in B \}$.  For $x\notin G$ we observe that the coset $xG$ of $G$ in $H$ is a permutation array 
with $hd(xG) = hd(G)$.  For cosets $x_{1}G, x_{1}G, \cdots, x_{k}G$ of $G$, the Hamming distance of the permutation array $\cup_{1\le i\le k}x_{i}G$ is the minimum of $d$  
and $m$, where $m =$ min$\{hd(x_{i}G, x_{j}G): 1\le i < j\le k\}$.  The method of coset search is to iteratively find coset representatives $x_{i}$ so that $m$, while in general less than $d$, 
is still reasonably large.  The partition and extension method is a way of obtaining constructive lower bounds $M(n+1, d+1)$ from such bounds for $M(n, d)$.  

Moving closer to the subject of this paper, we consider a class of optimal constructions which arise through sharply transitive groups.  
We say that a permutation array $A$ on a set $\Omega$ of size $n$ is {\it sharply $k$-transitive} on $\Omega$ 
if given any two $k$-tuples $x_{1}, x_{2}, \cdots, x_{k}$ and $y_{1}, y_{2}, \cdots, y_{k}$ of distinct elements of $\Omega$ there exists a unique $g\in A$ such that 
$g(x_{i}) = y_{i}$ for all $1\le i\le k$.  In our applications $A$ will be the set of image strings of a permutation group acting on $\Omega$.  
From the bound $M(n,d)\le \frac{n!}{(d-1)!}$ we have for any positive integer $k$ that 
$M(n, n-k+1)\le \frac{n!}{(n-k)!}$.  Now if $G$ is a sharply $k$-transitive group acting on $\Omega$, then $|G| = \frac{n!}{(n-k)!}$.  Also, in such 
a $G$ any two distinct elements $g,h$ of $G$ can agree in at most $k-1$ positions, since otherwise $gh^{-1}$ is a nonidentity element of $G$ 
fixing at least $k$ elements of $\Omega$, contrary to sharp $k$-transitivity.  Thus $hd(G) \geq n-k+1$.  So a sharply $k$-transitive group $G$ 
implies the existence of an optimal array (the set of image strings of elements of $G$) realizing $M(n,n-k+1) = \frac{n!}{(n-k)!}$.  The following theorem 
gives a strong converse to the above, including the generalization to arbitrary arrays that may not be groups.            

\begin{theorem} \label{Deza} \cite{BCD}  Let $A$ be a permutation array on a set of $n$ letters satisfying $hd(A)\geq n-k+1$.  Then 
$|A| = \frac{n!}{(n-k)!} = M(n, n-k+1) \Longleftrightarrow A$ is sharply $k$-transitive on this set. 
\end{theorem}   

The sharply $k$-transitive groups (for $k\geq 2$) are known, and these are as follows \cite{C}, \cite{DM}, \cite{Rob};

\medskip

\noindent $k = 2:$ the Affine General Linear Group $AGL(1,q)$ acting on the finite field $GF(q)$, consisting of the transformations $\{ x\rightarrow ax+b: x, a\ne 0, b\in GF(q) \}$,

\smallskip

\noindent $k = 3:$ the Projective Linear Group $PGL(2,q)$ acting on $GF(q)\cup \{\infty\}$, consisting of the transformations $\{ x\rightarrow \frac{ax+b}{cx+d}: x, a, b, c, d\in GF(q), ad-bc\ne 0\}$,

\smallskip

\noindent $k = 4:$ the Mathieu group $M_{11}$ acting on a set of size $11$, 

\smallskip

\noindent $k = 5:$ the Mathieu group $M_{12}$ acting on a set of size $12$, 

\smallskip

\noindent arbitrary $k:$ the symmetric group $S_{k}$ acting on a set of size $k$ is sharply $k$ and $(k-1)$-transitive, as well as the alternating group $A_{k}$ acting on a 
set of size $k$ is sharply $(k-2)$-transitive (\cite{Rob}, Theorem 7.1.4). 

In this paper we obtain new lower bounds on $M(n,d)$ for $n$ and $d$ near a prime power.  Previous results of this kind are given in \cite{CCD} where 
it is shown that for $n = 2^{k}$ with $n\not\equiv 1$(mod $3$) we have $M(n,n-3)\geq (n+2)n(n-1)$ and $M(n,n-4)\geq \frac{1}{3}n(n-1)(n^{2}+3n+8)$.  It is also shown 
that for any prime power $n$ with $n\not\equiv 2$(mod $3$) we have $M(n,n-2)\geq n^{2}$.  These results are based on permutation polynomials.  Similar such results appearing in 
\cite{BLS}, are based on a {\it contraction} operation applied to permutation arrays defined in the next section.  The latter results yield 
$M(n-1, n-3)\geq n^{2} - 1$ when $n\not\equiv 1$(mod $3$), and $M(n-2, n-5)\geq n(n-1)$ 
when $n\not\equiv 2$(mod $3$) and $n\not\equiv 0,1$(mod $5$).   

Our goal is to obtain new lower bounds when $q$ is prime power satisfying $q\equiv 1$ (mod $3$); that is, 
to resolve the case left open in \cite{BLS} where the methods of that paper do not apply.  
For such $q$, we accomplish this by applying the contraction operation to the permutation arrays $AGL(1,q)$ and $PGL(2,q)$ (acting on $GF(q)$ and on 
$GF(q)\cup \{\infty\}$ respectively),  
obtaining the following constructive lower bounds. 

\medskip

\noindent \textbf{1.} for $q\geq 7$, $M(q-1,q-3)\geq (q^{2} - 1)/2$ for $q$ odd and $M(q-1,q-3)\geq (q-1)(q+2)/3$ for $q$ even, and

\smallskip

\noindent \textbf{2.} for  $q\geq 13$, $M(q,q-3)\geq Kq^{2}\log q$ for some constant $K$ if $q$ is odd, and

\smallskip

\noindent \textbf{3.} bounds for $M(n,d)$ for a finite number of exceptional pairs $n,d$, obtained from the Mathieu groups.

We will use standard graph theoretic notation.  In particular for a graph $G$ and $S\subseteq V(G)$ we let $[S]_{G}$ be the graph 
with vertex set $S$ and edge set $E([S]_{G}) = \{ xy: x, y \in S, xy\in E(G) \}$, and we call it {\it the graph induced by $S$}.  When $G$ is understood by context, 
we abbreviate $[S]_{G}$ by $[S]$.  

\subsection{Contraction and the Contraction Graph} 

Consider a permutation array $A$ acting on a set $\Omega = \{ x_{1}, x_{2}, \cdots , x_{n}\}$ of size $n$, where the elements of $\Omega$ are 
ordered by their subscripts.  We distinguish some element, say $x_{n}$, by renaming it $F$.  Thus the image string of any element 
$\sigma\in A$ will be $\sigma(x_{1})\sigma(x_{2})\cdots\sigma(F)$, and we say that $\sigma(x_{i})$ occurs in {\it position} or {\it coordinate} $x_{i}$ 
of the string. 
 Now for any 
$\pi\in A$, define the permutation $\pi^{\triangle}$ on $\Omega$ by
 
\bigskip

\[   
\pi^{\triangle}(x) = 
     \begin{cases}
		\pi(F) &\quad\text{if } x = \pi^{-1}(F),\\
	F &\quad\text{if } x = F,\\ 
        \pi(x) &\quad\text{otherwise.}\\   
     \end{cases}
\]

\bigskip

Thus the image string of $\pi^{\triangle}$ is obtained from the image string of $\pi$ by 
interchanging the symbols $F$ and $\pi(F)$ if $\pi(F)\ne F$, while $\pi^{\triangle}=\pi$ if and only if $\pi(F) = F$.  In either case,   
 $\pi^{\triangle}$ has $F$ as its final symbol.  We let $\pi^{\triangle}_{-}$  be the permutation on $n-1$ symbols 
obtained from $\pi^{\triangle}$ by dropping the last symbol $F$ from $\pi^{\triangle}$.  As an example, if 
$\pi = aFbcd$, then $\pi^{\triangle} = adbcF$, and $\pi^{\triangle}_{-} = adbc$.  We call the operation $\pi\rightarrow \pi^{\triangle}_{-} $ 
{\it contraction}, and we call $\pi^{\triangle}_{-}$ the {\it contraction of the permutation $\pi$}.
Further, for any subset $R\subset A$, let 
$R^{\triangle} =  \{ \pi^{\triangle}: \pi\in R \}$, and $R^{\triangle}_{-} = \{ \pi^{\triangle}_{-}: \pi\in R \}$.  So $R^{\triangle}_{-}$ 
is a permutation array on the symbol set $\Omega - \{F\}$ of size $n-1$, and is called the {\it contraction of $R$}.

We note some basic properties related to the contraction operation.

\begin{lemma} \label{basicsintro} Let $G$ be a permutation group acting on the set $\Omega$ of size $n$, and let $\pi,\sigma\in G$. 

\noindent {\textbf a)} The only coordinates in either $\pi$ or $\sigma$ whose values are affected by the $\triangle$ operation are $\pi^{-1}(F), \sigma^{-1}(F),$ and $F$. 
Thus $hd(\pi^{\triangle}, \sigma^{\triangle})\geq hd(\pi, \sigma) - 3.$

\noindent {\textbf b)} Assume $hd(\pi^{\triangle}, \sigma^{\triangle}) = hd(\pi, \sigma) - 3$.  Then $\pi\sigma^{-1}$ contains 
a $3$-cycle in its disjoint cycle factorization, and $|G|$ is divisible by $3$.

\noindent {\textbf c)} Let $S\subseteq G$.  Then $|S^{\triangle}| = |S^{\triangle}_{-}|$ and $hd(S^{\triangle}) = hd(S^{\triangle}_{-})$.  If  
also $hd(S) > 3$,  then $|S| = |S^{\triangle}|$.

\end{lemma} 

\begin{proof} Part a) follows immediately from the definition of the $\triangle$ operation.

For b), the assumption implies that there are positions $x_{i}, x_{j}, F$ at which the image strings of $\pi$ and $\sigma$ disagree and $\pi^{\triangle}$ and $\sigma^{\triangle}$ 
agree.  So for some indices $s,t$ we must have $\pi(x_{i}) = x_{s}, \pi(x_{j}) = F, \pi(F) = x_{t}$, while $\sigma(x_{i}) = F, \sigma(x_{j}) = x_{t}, \sigma(F) = x_{s}$.  
Then $\pi\sigma^{-1}$ (composing left to right) contains 
the $3$-cycle $(x_{i}, F, x_{j})$ in its disjoint cycle factorization.  Thus the subgroup of $G$ generated by $\pi\sigma^{-1}$ has order divisible by $3$, and hence 
$|G|$ is divisible by $3$ by Lagrange's theorem. 

Consider c).  The first two equalities follow from the fact that all image strings in $S^{\triangle}$ have $F$ as their last coordinate.
To see $|S| = |S^{\triangle}|$ when $hd(S) > 3$, suppose to the contrary that $\pi^{\triangle} = \sigma^{\triangle}$ for distinct $\pi, \sigma\in S$.  As noted in the proof of part a), 
$\pi^{\triangle}$ and $\sigma^{\triangle}$ can agree in at most three positions where $\pi$ and $\sigma$ disagreed.  Thus $\pi$ and $\sigma$ already agreed in at 
least $n-3$ positions.  So $hd(\pi, \sigma)\le 3$, a contradiction.        \end{proof}

\smallskip

Consider a permutation array $H$ on $n$ symbols with $hd(H) = d$.  The array $H^{\triangle}_{-}$ is on $n-1$ symbols and satisfes $hd(H^{\triangle}_{-})\geq d-3$ 
by Lemma \ref{basicsintro}a,c.  For the arrays $H = AGL(1,q), PGL(2,q)$ and certain Mathieu groups, our goal is to find a subset $I\subset H$ with 
$hd(I^{\triangle}_{-})\geq d-2$; that is, a subset $I$ whose contraction $I^{\triangle}_{-}$ has Hamming distance larger by $1$ than the lower bound $d-3$ for $hd(H^{\triangle}_{-})$
 given by Lemma \ref{basicsintro}a.  
The lower bound $M(n-1, d-2)\geq |I^{\triangle}_{-}|$ follows.  Now our underlying arrays $H$ will satisfy $hd(H) > 3$, so $hd(I)\geq hd(H) > 3$.
Thus by Lemma \ref{basicsintro}c we have $hd(I^{\triangle}_{-}) = hd(I^{\triangle})$ and $|I^{\triangle}_{-}| = |I^{\triangle}| = |I|$.  So we get $M(n-1, d-2)\geq |I|$, yielding the main results of this paper.

To find such a subset $I$ of $H$, we employ a graph $C_{H}$ defined as follows.

\begin{definition} Let $H$ be a permutation array with $hd(H) = d$.  Define the \underline{contraction graph for $H$}, denoted $C_{H}$, by 
$V(C_{H}) = H$ and $E(C_{H}) = \{ \pi\sigma: \pi, \sigma \in H, hd(\pi^{\triangle}, \sigma^{\triangle}) = d-3\}$.  
\end{definition}
 For $\pi\in C_{H}$, notice that if $\pi(F) = F$, then $\pi$ is an isolated 
point in $C_{H}$.  This is because then $\pi^{\triangle} = \pi$, so that for any other $\sigma\in C_{H}$ we have 
$hd(\pi^{\triangle}, \sigma^{\triangle}) = hd(\pi, \sigma^{\triangle}) \geq hd(\pi, \sigma) - 2$, implying no edge joining $\pi$ 
and $\sigma$ in $C_{H}$.  We thus have the following characterization of edges in $C_{H}$:

\begin{equation} \label{contractionedges}
\pi\sigma\in E(C_{H})\Longleftrightarrow \{ \pi(F)\ne F, \sigma(F)\ne F, \sigma(\pi^{-1}(F)) = \pi(F), \pi(\sigma^{-1}(F)) = \sigma(F) \}.
\end{equation}

This condition on edges is illustrated in Figure 1.

\begin{figure}[htb]
\begin{center}
\includegraphics{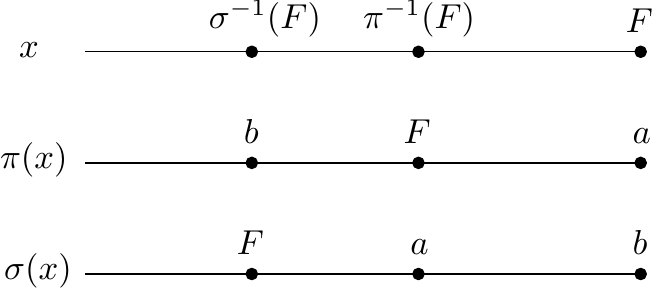}
\end{center}
\caption{Neighbors $\pi, \sigma$ in $C_{H}$; $\sigma(\pi^{-1}(F)) = \pi(F)$, and $\pi(\sigma^{-1}(F)) = \sigma(F)$}
\label{f4}
\end{figure}

Since $hd(H^{\triangle})\geq d-3$ 
by Lemma \ref{basicsintro}a, it follows that
any independent set $I$ of vertices in $C_{H}$ must satisfy $hd(I^{\triangle})\geq d-2$.  Now by 
using Lemma \ref{basicsintro}a,c (together with $hd(H) > 3$ for our arrays $H$) 
 get $M(n-1, d-2)\geq |I|$ as explained in the preceding paragraph. 
 We are thus reduced to finding a large independent set in $C_{H}$ for each of 
the arrays $H = AGL(1,q), PGL(2,q)$, and Mathieu groups considered in this paper.

 \section{The contraction graph for $AGL(1,q)$}
        
Let $q$ be a prime power.  Recall the Affine General Linear Group $AGL(1,q)$ acting as a permutation group on the finite field $GF(q)$ of size $q$, as the set 
of transformations $\{ x\rightarrow ax+b: a\ne 0, x, b\in GF(q) \}$ under the binary operation of composition.  
For any $\pi\in AGL(1,q)$ the permutation $\pi^{\triangle}$ on $GF(q)$ is defined as in the previous section, 
based on some ordering $x_{1}, x_{2}, \cdots , x_{q}$ of the elements of $GF(q)$, where $F = x_{q}$ is 
a distinguished element. 
Clearly $|AGL(1,q)| = q(q-1)$.  By standard facts $AGL(1,q)$ is sharply $2$-transitive in this action, and it is straightforward to see that 
$hd(AGL(1,q)) = q-1$ (see Lemma \ref{roots}a for most of that short proof).    

Our goal in this section 
is to obtain a lower bound on $M(q-1,q-3)$ for prime powers 
$q\geq 7$ satisfying $q\equiv 1($mod $3)$.  Our method will involve the contraction graph $C_{AGL(1,q)}$ for $AGL(1,q)$, which we 
henceforth abbreviate by $C_{A}(q)$.

By definition we then have 
$V(C_{A}(q)) = AGL(1,q)$, and $E(C_{A}(q)) = \{ \pi\sigma: hd(\pi^{\triangle}, \sigma^{\triangle}) =  q - 4 \}$.  Following the plan 
described in the previous section, we find an independent set $I$ in $C_{A}(q)$.  Once we have such an $I$, then 
$I^{\triangle}_{-}$ is a permutation array on $q-1$ symbols, and by Lemma \ref{basicsintro}c satisfies $hd(I^{\triangle}_{-}) = hd(I^{\triangle})\geq q-3$.  
This implies the lower bound 
 $M(q-1,q-3) \geq |I^{\triangle}_{-}| = |I^{\triangle}| = |I|$, the last equality by Lemma \ref{basicsintro}c, since $q\geq 7$ implies $hd(I) \geq q-1 > 3$.  
 The actual size of $I$ will then yield our precise lower bound. 
 
 We are thus reduced to finding a large independent set $I$ in $C_{A}(q)$, and from this we get the bound $M(q-1,q-3)\geq |I|$.  
We begin on that in the following Lemma, which establishes relations in the the finite field 
$GF(q)$ that correspond to edges in the graph $C_{A}(q)$.

\begin{lemma} \label{roots} Let $\pi$ and $\sigma$ be vertices of the graph $C_{A}(q)$, $q\equiv 1($mod $3)$, say 
with $\sigma(x) = ax + r$ and $\pi(x) = bx + s$.

\noindent \textbf{a)} If $a\ne b$, then $hd(\pi, \sigma) = q - 1$. 

\noindent \textbf{b)} If $\pi(F) = F$, then $\pi$ is an isolated point in $C_{A}(q)$.  There are $q-1$ points $\pi$ satisfying $\pi(F) = F$.     

\noindent \textbf{c)} Suppose $\pi$ and $\sigma$ are neighbors in $C_{A}(q)$. Then 

\textbf{c1)} $hd(\pi, \sigma) = q-1$, and $hd(\pi^{\triangle}, \sigma^{\triangle}) = hd(\pi, \sigma) - 3$, and

\textbf{c2)} $\frac{a}{b}$ and $\frac{b}{a}$ 
are the distinct roots of the quadratic $t^{2} + t + 1 = 0$ over $GF(q)$.

 \end{lemma} 

\begin{proof} For a),  just observe that $\pi(x) = \sigma(x)$ has the unique solution $x = \frac{s - r}{a - b}$.

For the first claim in b) suppose not, and let $\sigma$ be a neighbor of $\pi$ in $C_{A}(q)$.  
Then we have $hd(\pi^{\triangle}, \sigma^{\triangle}) = q - 4$, implying also 
that $hd(\pi, \sigma) = q - 1$ by Lemma \ref{basicsintro}a.  Let $i$ be the coordinate of agreement between $\pi$ and $\sigma$.  Since 
$\pi(F) = F$, we have $\pi^{\triangle} = \pi$.  Thus $hd(\pi, \sigma^{\triangle}) = q - 4$.  Now $\sigma^{\triangle}$ can have at most two 
coordinates, apart from $i$, in which it agrees with $\pi$, these being $F$ and $j = \sigma^{-1}(F)$.  So altogether 
$\pi$ and $\sigma^{\triangle}$ agree in at most the $3$ coordinates $i,j, F$.  So 
$ q - 4 = hd(\pi, \sigma^{\triangle})  \geq q-3$, a contradiction. 

Now consider the second claim in b).  For any fixed $i\in GF(q)$, $i\ne F$

Since $\pi(F) = F$,  we have $q-1$ choices for the value $\pi(i)$ for any fixed $i\in GF(q)$, $i\ne F$.  
Hence there are $q-1$ choices for the ordered pair $( \pi(F) (= F), \pi(i) )$, each such choice determining $\pi$ uniquely by the sharp $2$-transitivity of $AGL(1,q)$ acting on $GF(q)$. The claim follows.

For c1), by the definition of edges in $C_{A}(q)$ we have $q-4 = hd(\pi^{\triangle}, \sigma^{\triangle}) \geq hd(\pi, \sigma) - 3$ using Lemma \ref{basicsintro}a.  
Since $hd(\pi, \sigma) = q$ or $q-1$, it follows that 
$hd(\pi, \sigma) = q-1$ and we have equality throughout, as required.

Consider c2).  By part c1) we have $hd(\pi, \sigma)= q-1$   
and $hd(\pi^{\triangle}, \sigma^{\triangle}) = hd(\pi, \sigma) - 3$.  
So there are distinct $\alpha, \beta\in GF(q)$, with neither $\alpha$ nor $\beta$ being $F$, 
such that $\sigma(F) = i$, $\sigma(\alpha) = F$, and $\sigma(\beta) = j$, and $\pi(F) = j$, $\pi(\alpha) = i$, and $\pi(\beta) = F$ for distinct $i,j\in GF(q)$.  This gives the following 
set of equations in $GF(q)$.

\begin{equation} \label{sysa}
\left\{
\begin{array}{rcl}
\sigma(\alpha)-\sigma(\beta)=& F - j &= a(\alpha-\beta) \\
\sigma(\alpha)-\sigma(F) =&F - i &= a(\alpha - F) \\
\pi(\alpha)-\pi(\beta) =&i - F &= b(\alpha-\beta) \\
\pi(\alpha)-\pi(F) =&i - j &= b(\alpha - F).\\
\end{array}
\right.
\end{equation}

The second and third equations of (\ref{sysa}) imply

\begin{equation} \label{sysb}
a(\alpha-F)=-b(\alpha-\beta).
\end{equation}

Now starting with the first equation of (\ref{sysa}) we obtain
\begin{align*} 
a(\alpha-\beta) &= F - j  \\
&= (F - i)+(i-j) \\
&= (a+b)(\alpha-F) \text{\quad (by the second and fourth equations of (\ref{sysa}))}.
\end{align*}

Multiplying equation (\ref{sysb}) by $a$ and the last equation by $b$, we obtain the equations
\begin{equation} \label{sysc}
\left\{
\begin{array}{rl}
a^2(\alpha-F) &= -ab(\alpha-\beta) \\
ab(\alpha-\beta) &= b(a+b)(\alpha-F) .\\
\end{array}
\right.
\end{equation}

Thus $a^2(\alpha-F) = -b(a+b)(\alpha-F)$, and on dividing by $\alpha-F$ (since $\alpha\ne F$) we obtain
\begin{align} \label{quad}
a^2+b(a+b)=0.
\end{align}
Dividing equation (\ref{quad}) by $a^2$ or by $b^2$, we obtain that $a/b$ and $b/a$ are both roots of the equation $t^2+t+1=0$.    

We show that $a/b$ and $b/a$ are distinct.   
Assuming otherwise, then $a/b = 1$ or $-1$.  If $q$ is even then from $1+ t + t^{2} = 0$ we get the contradiction $1 = 0$ since the characteristic is $2$.  
Now assume $q$ is odd.  If $a/b = 1$, then we get $1+1+1=0$, forcing $q\equiv 0($mod $3)$, a contradiction.  If $a/b = -1$, then we 
get $1=0$, again a contradiction.   \end{proof}

\bigskip

Let $t_{1}$ and $\frac{1}{t_{1}}$ be the distinct roots of $t^{2} + t + 1 = 0$ in $GF(q)$ for $q\equiv 1 \pmod 3$ (by Corollary \ref{quadroots}a).   
Let $\pi\in C_{A}(q)$ with $\pi(x) = ax + r$, and let $\sigma$ be a neighbor of $\pi$ in $C_{A}(q)$.  Then by 
Lemma \ref{roots}c2 we have $\sigma(x) = at_{1} + s_{1}$ 
or $\sigma(x) = a\frac{1}{t_{1}} + s_{2}$, so far with $s_{1}$ and $s_{2}$ undetermined.              
The next lemma shows that $s_{1}$ and $s_{2}$ are uniquely determined by $\pi$ and $t_{1}$.

\begin{lemma}\label{degree2} Let $q$ be a prime power with $q\equiv 1 \pmod 3$. 
Suppose $\pi$ is not an isolated point of $C_{A}(q)$, say with $\pi(x) = ax + r$.  Let $t_{1}$ be a root of $t^{2} + t + 1 = 0$ in $GF(q)$. 
Then the neighbors of $\pi$ 
in $C_{A}(q)$ are $\sigma_{1}$ and $\sigma_{2}$, given by $\sigma_{1}(x) = at_{1}x + (a - t_{1})F + r(1 + t_{1})$ and 
$\sigma_{2}(x) = a\frac{1}{t_{1}}x + (a - \frac{1}{t_{1}})F + r(1 + \frac{1}{t_{1}})$.  In particular, each 
non-isolated point of $C_{A}(q)$ has degree $2$ in $C_{A}(q)$.  
\end{lemma}

\begin{proof}: Let $N(\pi)$ be the set of neighbors of $\pi$ in $C_{A}(q)$.  First we verify that $\sigma_{1}, \sigma_{2}\in N(\pi)$, giving details only 
for $\sigma_{1}\in N(\pi)$ as the containment $\sigma_{2}\in N(\pi)$ is proved similarly.  To do this, we show that all conditions of (\ref{contractionedges}) are 
satisfied with $\sigma_{1}$ playing the role of $\sigma$.  Clearly $\pi(F)\ne F$ since $\pi$ is not isolated.  To show $\sigma_{1}(F)\ne F$, assume not.  
Suppose first that $a\ne 1$.  Then $\sigma_{1}(F) = F$ yields $F = \frac{r}{1-a}$.  
But now we get $\pi(F) = \frac{ar}{1-a} + r = \frac{r}{1-a} = F$, a contradiction.  Next suppose $a = 1$ so $\pi(x) = x + r$.  Then $\sigma_{1}(F) = F$ together 
with $a = 1$ yields $r(1+t_{1}) = 0$. 
Combining this  with $t_{1}\ne -1$ yields 
$r = 0$.  But then $\pi(F) = F$, a contradiction.  

So it remains to show that $\sigma_{1}(\pi^{-1}(F)) = \pi(F)$ and 
$\pi(\sigma_{1}^{-1}(F)) = \sigma_{1}(F)$.  For the first equality, solving $ax + r = F$ we 
obtain $\pi^{-1}(F) = \frac{F-r}{a}$.  Thus $\sigma_{1}(\pi^{-1}(F)) = at_{1}\big(\frac{F-r}{a} \big) + (a-t_{1})F + r(1 + t_{1}) 
 = aF + r = \pi(F)$, as required.  For the second equality, from the formula for $\sigma_{1}$ 
 we obtain $\sigma_{1}^{-1}(F) = \frac{1}{at_{1}}\big( F(1-a+t_{1}) - r(1+t_{1})\big)$.  Plugging this into 
 $\pi$ and simplifying, we obtain $\pi(\sigma_{1}^{-1}(F)) = \frac{1}{t_{1}}\big( F(1-a+t_{1}) - r(1+t_{1})\big) + r$.  Working backwards from the equality 
 $\pi(\sigma_{1}^{-1}(F)) = \sigma_{1}(F)$ we must show that $\frac{1}{t_{1}}\big( F(1-a+t_{1}) - r(1+t_{1})\big) = F\big(a-t_{1}+at_{1}\big) + rt_{1}.$  This is equivalent to 
 $F\big(1-a+t_{1} \big) = r(t_{1}^{2}+t_{1}+1) + F(at_{1}-t_{1}^{2}+at_{1}^{2}) = F(at_{1}-t_{1}^{2}+at_{1}^{2})$.  We are now reduced to showing 
 $1-a+t_{1} = at_{1}-t_{1}^{2}+at_{1}^{2}.$  This follows from $t_{1}^{2} + t_{1} + 1 = 0$. 
 
 Now let $\sigma\in N(\pi)$, and we show that $\sigma = \sigma_{1}$ or $ \sigma_{2}$.  By Lemma \ref{roots}, we know that 
 $\sigma(x) = at_{1} + s_{1}$ or $\sigma(x) = a\frac{1}{t_{1}} + s_{2}$ for suitable $s_{1}, s_{2}\in GF(q)$.  Suppose first 
 that  $\sigma(x) = at_{1} + s_{1}$.  Applying the equality $\sigma(\pi^{-1}(F)) = \pi(F)$ together with $\pi^{-1}(F) = \frac{F-r}{a}$, we get 
 $at_{1}\big(\frac{F-r}{a}\big) + s_{1} = aF + r$. so $s_{1} = (a-t_{1})F + r(1+t_{1})$.  Thus $\sigma = \sigma_{1}$.  
A very similar argument shows that if $\sigma(x) = a\frac{1}{t_{1}} + s_{2}$, then $s_{2} = (a-\frac{1}{t_{1}})F + r(1+ \frac{1}{t_{1}})$, 
and thus $\sigma = \sigma_{2}$.  So we have $N(\pi) = \{\sigma_{1}, \sigma_{2}\}$, completing the proof.           \end{proof} 

Consider the subgroup $Q = \{ x + b: b\in GF(q)\}$ of $AGL(1,q)$.  Clearly $|Q| = q$, and for each $h\in GF(q), h\ne 0$, $Q$ has the coset 
$hxQ = \{ hx + b: b\in GF(q) \}$, which we abbreviate by $Q_{h}$.

\begin{theorem} Let $q$ be a prime power with $q\equiv 1 \pmod 3$.  Then the connected components of $C_{A}(q)$ are as follows.

\noindent{\bf a)} There are $q-1$ isolated points, these being the points $\pi$ satisfying $\pi(F) = F$. 

\noindent{\bf b)} If $q$ is odd, then each non-isolated point component is a cycle of length 6.

\noindent{\bf c)} If $q$ is even, then each non-isolated point component is a cycle of length 3.

\end{theorem}

\begin{proof} For part a), we show that $\pi\in C_{A}(q)$ is an isolated point if and only if $\pi(F) = F$.  Then a) follows by Lemma \ref{roots}b.  

If $\pi(F) = F$, then immediately $\pi$ is isolated in $C_{A}(q)$ by Lemma \ref{roots}b.  For the converse, suppose to the contrary that 
$\pi$ is isolated and that $\pi(F)\ne F$.  Let $\sigma_{1}$ be given by $\sigma_{1}(x) = at_{1}x + (a - t_{1})F + r(1 + t_{1})$ as in 
Lemma \ref{degree2}.  Then the proof of Lemma  \ref{degree2}, starting from the established claim $\pi(F)\ne F$ (this claim being an assumption here), shows 
that $\sigma_{1}$ is a neighbor of $\pi$ in $C_{A}(q)$.  This contradicts $\pi$ being isolated.             

Consider part b).  By Lemma \ref{degree2} each nontrivial component of $C_{A}(q)$ is a cycle.  Let $\pi_{0}$ be a point on such a cycle $C$, say with $\pi_{0}\in Q_{a}$.  
Let $t_{1}$ be a fixed root of $t^{2} + t + 1 = 0$.  Consider 
a sequence of $4$ vertices $\pi_{0}\pi_{1}\pi_{2}\pi_{3}$ on $C$ with $\pi_{j}\pi_{j+1}\in E(C_{A}(q))$ for $0\le j\le 2$.  
We may suppose that $\pi_{j}\in Q_{at_{1}^{j}}$ by Lemma \ref{degree2} and straightforward induction (otherwise replace $t_{1}$ by $\frac{1}{t_{1}}$).  
Thus $\pi_{0}, \pi_{1}, \pi_{2}$ are distinct since they belong  distinct cosets of $Q$.  Since $t_{1}^{3} = 1$, we see also that $\pi_{0}$ and $\pi_{3}$ 
belong to the same coset $Q_{a}$ of $Q$.  We now show that $\pi_{0}\ne \pi_{3}$.    
Writing $\pi_{1}(x) = bx + c$ (so $b = at_{1}$), we 
apply the first and third equations of (\ref{sysa}) with $\pi_{0}$ and $\pi_{1}$ playing the roles of 
$\sigma$ and $\pi$ respectively, to  get $t_{1} = \frac{b}{a} = -\big(\frac{i-F}{j-F}\big) = -\big(\frac{\pi_{0}(F)-F}{\pi_{1}(F)-F}\big)$.  Applying this equation two more times 
we get $1 = t_{1}^{3} = -\big(\frac{\pi_{0}(F)-F}{\pi_{3}(F)-F}\big)$, so that $\big(\frac{\pi_{0}(F)-F}{\pi_{3}(F)-F}\big) = -1$.  Thus $\pi_{0}(F)\ne \pi_{3}(F)$ , 
so $\pi_{0}\ne \pi_{3}$. Thus each cycle component has length at least $4$.

Consider now a sequence of $7$ vertices $\pi_{0}\pi_{1}\cdots \pi_{6}$ on $C$ with $\pi_{j}\pi_{j+1}\in E(C_{A}(q))$ for $0\le j\le 5$. We claim 
the first $6$ of these $\pi_{0}, \pi_{1}, \cdots, \pi_{5}$ must be distinct as follows.  Clearly any two vertices $\pi_{j}, \pi_{j+3}$ are distinct, $0\le j\le 2$, by the same 
argument that showed $\pi_{0}\ne \pi_{3}$.  But any two vertices $\pi_{i}, \pi_{j}$ with $i\not\equiv j ($mod $3)$ are distinct, 
since $t_{1}\ne 1$ and $t_{1}^{2}\ne 1$ imply that they  belong to different 
cosets of $Q$, proving the claim.  Finally note that $1 = t_{1}^{6} = \big(\frac{\pi_{0}(F)-F}{\pi_{6}(F)-F}\big)$, so that $\pi_{0}(F) = \pi_{6}(F)$.  Since also 
$\pi_{0}$ and $\pi_{6}$ also belong to the same coset $Q_{a}$ of $Q$, it follows that $\pi_{0} = \pi_{6}$.  Thus the component $C$ containing $\pi_{0}$ 
is a cycle of length $6$, as required.

Now consider part c).  Consider as above the sequence of $4$ vertices $\pi_{0}\pi_{1}\pi_{2}\pi_{3}$ in a nontrivial component, with 
$\pi_{j}\pi_{j+1}\in E(C_{A}(q))$ for $0\le j\le 2$.  We get $\big(\frac{\pi_{0}(F)-F}{\pi_{3}(F)-F}\big) = -1 = 1$ since $q$ is even.  So since also $\pi_{0}$ and $\pi_{3}$ belong to 
the same coset $Q_{a}$ of $Q$, it follows that $\pi_{0} = \pi_{3}$.  Thus the cycle containing $\pi_{0}$ has length $3$.     \end{proof}

\begin{corollary} Let $q$ be a prime power with $q\equiv 1 \pmod 3$ and $q\geq 7$.  Then

\smallskip 

\noindent{\bf a)} if $q$ is odd, then $M(q-1, q-3)\geq (q^2-1)/2$, and 

\smallskip

\noindent{\bf b)} if $q$ is even, then $M(q-1, q-3)\geq (q-1)(q+2)/3$.

\end{corollary}

\begin{proof} For part a) we form an independent set $I$ in $C_{A}(q)$ by taking $3$ independent points in each cycle component of 
of length $6$, together with the set $Y$ of isolated points.  Then $M(q-1, q-3)\geq |Y| + \frac{1}{2}( |C_{A}(q) - Y| ) = q-1 + \frac{1}{2}\big(q(q-1) - (q-1)\big) = (q^2-1)/2$ , as required. 

For part b), we form an independent set $I$ in $C_{A}(q)$ by taking one point from each length $3$ cycle component, together with the set $Y$ of isolated points.  
We then have $M(q-1, q-3)\geq |Y| + \frac{1}{3}( |C_{A}(q) - Y| ) = (q-1)(q+2)/3$.

\end{proof}

The lower bounds in this corollary should be compared to the lower bound $M(q,q-2)\geq q^{2}$ for prime powers $q\not\equiv 2$(mod $3$), derived by 
using permutation polynomials \cite{CCD}.

 \section{The contraction graph for $PGL(2,q)$}

 Let $q$ be a power of a prime.  The permutation group $PGL(2,q)$ is defined as the set of one to one functions 
 $\sigma: GF(q)\cup \{\infty\} \rightarrow GF(q)\cup \{\infty\}$, under the binary operation of composition, given by
\begin{equation} \label{eqn:PGLstandard}
\{~ \sigma(x) = \frac{ax+b}{cx+d} ~:~a,b,c,d \in GF(q), ad \neq bc, x \in GF(q) \cup \{\infty\}~ \}.
\end{equation}

\noindent
Here $\sigma(x)$ is computed by the rules: 
\begin{enumerate} 
\item if $x \in GF(q)$ and $x\neq-(d/c)$, then $\sigma(x) = \frac{ax+b}{cx+d} $, 
\item if $x \in GF(q)$ and $x = -(d/c)$, then $\sigma(x) = \infty$, 
\item if $x = \infty$, and $c \neq 0$, then $\sigma(x) = a/c$, and 
\item if $x = \infty$, and $c = 0$, then $\sigma(x) = \infty$.  
\end{enumerate}

We regard $PGL(2,q)$ as a permutation group acting on the set $GF(q)\cup \{\infty\}$ of size $q+1$ via the one to one map $x\mapsto \sigma(x)$.  
One can show that $|PGL(2,q)| = (q+1)q(q-1)$, and it is well known that 
$PGL(2,q)$ is sharply $3$-transitive in its action on $GF(q)\cup \{\infty\}$ (see \cite{Stin} for a proof).  It is straightforward to verify that 
$hd(PGL(2,q)) = q-1$, and by Theorem \ref{Deza} we have $M(q+1, q-1) = |PGL(2,q)| = (q+1)q(q-1)$.

Take a fixed ordering of 
$GF(q)\cup \{\infty\}$ with $\infty$ as final symbol, say $x_{1}, x_{2}, \cdots, x_{q}, \infty$ where the $x_{i}$ are the distinct elements 
of $GF(q)$.  Then any element $\pi\in PGL(2,q)$ is identified with the length $q+1$ string $\pi(x_{1})\pi(x_{2})\pi(x_{3})\cdots \pi(x_{q})\pi(\infty)$, which 
again we call the {\it image string} of $\pi$.  For any such $\pi\in PGL(2,q)$ the permutation $\pi^{\triangle}$ on 
$GF(q)\cup \{\infty\}$ is defined as in the section introducing contraction, where $F = \infty$ is 
the distinguished element of $GF(q)\cup \{\infty\}$ in that definition.   
As an example, if 
$\pi = a\infty bcde$, then $\pi^{\triangle} = aebcd\infty$, and $\pi^{\triangle}_{-} = aebcd$.    
In the same way, for any subset $R\subset PGL(2,q)$, the sets
$R^{\triangle}$, and $R^{\triangle}_{-}$ are defined as in that section, with $F = \infty$.  Since  $hd(PGL(2,q)) = q-1 = q+1-2$, the image strings of 
any two elements of $PGL(2,q)$ agree in at most two positions.  It follows from Lemma \ref{basicsintro}a that for any $\pi, \sigma\in PGL(2,q)$ we have 
$hd(\pi^{\triangle}, \sigma^{\triangle})\geq hd(\pi, \sigma) - 3 \geq q - 4$.  That is, $\pi^{\triangle}$ and $ \sigma^{\triangle}$ can agree in at most 
$5$ positions; up to 2 occurring from the original $\pi$ and $\sigma$, and up to 3 more occurring from the $\pi^{\triangle}$ and $\sigma^{\triangle}$ operation.

As noted earlier, lower bounds for $M(q,q-3)$ and $M(q,q-4)$ when $q\not\equiv 1($mod $3)$ based on permutation polynomials are known \cite{CCD}. 
Thus in this section, we restrict ourselves to the case $q\equiv 1($mod $3)$, $q$ an odd prime power, 
where such bounds are not known.
For technical reasons we take $q\geq 13$.

The plan will be similar in some respects 
to the one we used in the previous section.  That is, for a certain set $I\subset PGL(2,q)$ we will find a permutation array 
$I^{\triangle}_{-} \subset PGL(2,q)^{\triangle}_{-} $ on $q$ symbols with $hd(I^{\triangle}_{-})\geq q-3 $, thus obtaining the lower bound 
on $M(q,q-3)\geq |I^{\triangle}_{-} |$.  This set $I$ will be an independent set in the contraction graph $C_{PGL(2,q)}$ for $PGL(2,q)$, which we abbreviate by $C_{P}(q)$.

Since $hd(PGL(2,q)) = q-1$, $C_{P}(q)$ is given by $V(C_{P}(q)) = PGL(2,q)$, and $E(C_{P}(q)) = \{ \pi\sigma: hd(\pi^{\triangle}, \sigma^{\triangle}) =  q - 4 \}$.  
So edges $\pi\sigma$ of $C_{P}(q)$ correspond to pairs $\pi, \sigma\in PGL(2,q)$ for which $hd(\pi^{\triangle}, \sigma^{\triangle})$ achieves its 
least possible value of $q-4$, occurring when $\pi^{\triangle}$ and $\sigma^{\triangle}$ agree in $5$ positions, so consequently $hd(\pi^{\triangle}, \sigma^{\triangle}) = hd(\pi, \sigma) - 3$.  
Thus a set $I\subseteq V(C_{P}(q))$ is independent in $C_{P}(q)$ if and only if it satisfies $hd(I^{\triangle})\geq q-3$.  By 
Lemma \ref{basicsintro}c, we get $hd(I^{\triangle}_{-}) = hd(I^{\triangle})\geq q-3$, while $|I^{\triangle}_{-}| = |I^{\triangle}| = |I|$, 
with the last equality following from $hd(I^{\triangle}) = q-3 > 3$ since $q\geq 13$.         
 
We are thus reduced to finding an independent set $I$ in $C_{P}(q)$, from which $M(q, q-3)\geq |I|$ follows.          
 
To do this, it will be useful to represent functions in $PGL(2,q)$ in a form different than the standard $\frac{ax+b}{cx+d}$ form. 

\begin{definition} Fix a prime power $q$.

\noindent {\bf1.} Let $K, r, i\in GF(q)$ with $r\ne 0$.  Define the function $f:GF(q)\cup\{\infty\}\rightarrow GF(q)\cup\{\infty\}$ by 
$f(x) = K + \frac{r}{x-i}$ for $x\notin \{ i, \infty \}$, while $f(\infty) = K$ and $f(i) = \infty$. 

\smallskip

\noindent {\bf2.} Let $P = \{ K + \frac{r}{x-i}: K,r,i\in GF(q), r\ne 0 \}$ be the 
set of all functions defined in 1.   

\smallskip

\noindent{\bf3.} Let $N\subset PGL(2,q)$ be given by $N = \{ \pi\in PGL(2,q): \pi(x) = \frac{ax+b}{cx+d}, c\ne 0 \}$.   

\end{definition}

We will now see that $P$ is the same set of functions as $N$.

\begin{lemma} \label{format} Let  the map $\alpha: N\rightarrow P$ be defined as follows.  For any $\pi\in N$ with $\pi(x) = \frac{ax+b}{cx+d}$, let 
$\alpha(\pi)\in P$ be given by  $\alpha(\pi)(x) = \frac{a}{c} + \frac{\frac{bc-ad}{c^{2}}}{x + \frac{d}{c}}.$ Then 

\noindent {\bf a)} $\pi$ and $\alpha(\pi)$ are the same function on $GF(q)\cup\{\infty\}$.

\noindent {\bf b)} $|P| = |N| = q^{2}(q-1)$.

\noindent {\bf c)} The map $\alpha$ is one to one and onto.

\end{lemma}

\begin{proof} For a), straightforward manipulation shows that for $x\ne -\frac{d}{c}$ we have $\pi(x) = \frac{a}{c} + \frac{\frac{bc-ad}{c^{2}}}{x + \frac{d}{c}} = \alpha(\pi)(x)$.  
Also by definition $\alpha(\pi)(\infty) = \frac{a}{c} = \pi(\infty)$ and 
$\alpha(\pi)(-\frac{d}{c}) = \infty = \pi(-\frac{d}{c})$. so $\pi$ and $\alpha(\pi)$ are the same function. 

Consider b).  Clearly $|P| =  q^{2}(q-1)$ since there are $q-1$ 
choices for $r$, and $q$ choices for each of $K$ and $i$, independent of each other.  
To show  $|N| = q^{2}(q-1)$, observe first that for any $\pi\in PGL(2,q)$ with $\pi(x) = \frac{ax+b}{cx+d}$ we have 
$c = 0 \Leftrightarrow \pi(\infty) = \infty$.  The $\Rightarrow$ direction is immediate by definition.  
To see $\Leftarrow$, assume $c\ne 0$.  Then $\pi(\infty) = \frac{a}{c} \ne \infty$, completing the proof of 
the observation.  Next we have $\pi(\infty) = \infty  \Leftrightarrow \pi(x) = Ax + B\in AGL(1,q)$ for all $x$ for suitable 
$A\ne 0,B\in GF(q)$, by definition of computing in $PGL(2,q)$.  Thus we have 
$|N| = |PGL(2,q)| - |AGL(2,q)| = (q+1)q(q-1) - q(q-1) = q^{2}(q-1)$.

Part c) is immediate from parts a) and b), since any two distinct elements of $N$ are distinct as functions.  As an alternative (constructive) proof, 
let $f(x) = K + \frac{r}{x-i}\in P$ be given.  Then for $\pi(x) = \frac{Kx + r-iK}{x - i}\in N$ we have $\alpha(\pi) = f$.  Thus $\alpha$ is onto, and since 
$|P| = |N|$, $\alpha$ is also one to one.

\end{proof}

We now see how the above observations, together with results which come later, reduce the study of 
$C_{P}(q)$ to the set $P$ of permutations..  It was shown above that 
for $\pi\in PGL(2,q)$, we have $\pi(\infty) = \infty \Leftrightarrow c = 0$.  By condition (\ref{contractionedges}) (with $\infty = F$) we see that $\pi(\infty) = \infty$ 
implies that $\pi$ is an isolated point in $C_{P}(q)$, and we will see later that for $C_{P}(q)$ the converse is also true.  
So to study the structure of $C_{P}(q)$ apart from its isolated points, we 
are reduced to studying its subgraph induced by the permutations in $N$.  
By the bijection $\alpha: N\leftrightarrow P$, under which $\pi\in N$ and $\alpha(\pi)\in P$ are the 
same permutation on $GF(q)\cup \{\infty\}$, we are then reduced to studying $P$.

\begin{figure}[htb]
\begin{center}
\includegraphics{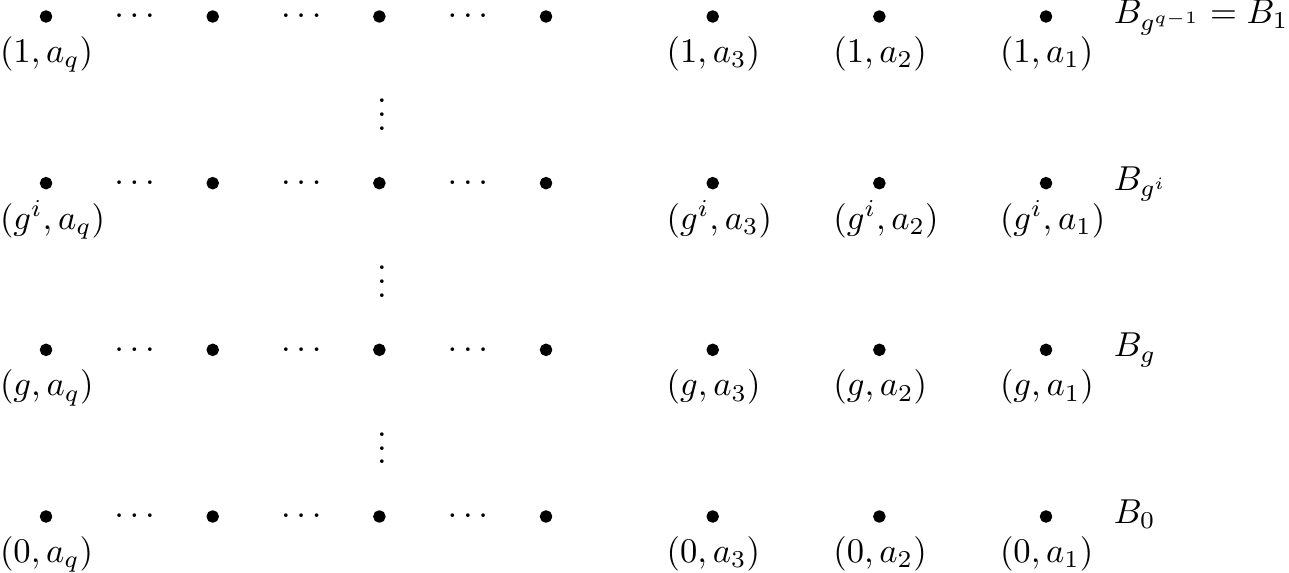}
\end{center}
\caption{The graph $P_1$, partitioned into levels $B_0$ and $B_{g^i},1\le i\le q-1$.}
\label{f1}
\end{figure}

 \begin{lemma} \label{associates1} Let $\pi, \sigma\in P$ with $\pi(x) = a + \frac{r}{x-i}$, $\sigma(x) = b + \frac{s}{x-j}$ with $r,s\ne 0$.  Then 
$hd(\pi^{\triangle}, \sigma^{\triangle}) = hd(\pi,\sigma) - 3 \Longleftrightarrow (b-a)(j-i) = r$ and $r = s$ .

\end{lemma}

\begin{proof} $\Longrightarrow$ : By assumption we have $\pi(\infty) = a$ and $\pi(i) = \infty$, together with $\sigma(j) = \infty$ and $\sigma(\infty) = b$. 
By Lemma \ref{basicsintro}a the only coordinates of either $\pi$ or $\sigma$ whose values are affected by the $\triangle$ operation are the $3$ coordinates 
$\pi^{-1}(\infty) = i, \sigma^{-1}(\infty) = j$, and $\infty$.  So the assumption $hd(\pi^{\triangle}, \sigma^{\triangle}) = hd(\pi,\sigma) - 3$ implies that 
$\sigma(i) = a$ and $\pi(j) = b$.  Thus we get   
$\pi(j) = a + \frac{r}{j-i} = b$, yielding $(b-a)(j-i) = r$ as required.  Now interchanging the roles of $\pi$ and $\sigma$ in this argument, specifically, 
using $\sigma(i) = b + \frac{s}{i-j} = a$, we get $(a-b)(i-j) = s$, so also $r = s$.  

$\Longleftarrow$ : Again by assumption we have $\pi(\infty) = a$, $\sigma(\infty) = b$, $\pi(i) = \infty$, $\sigma(j) = \infty$, and $(b-a)(j-i) = r$ .  To prove 
$hd(\pi^{\triangle}, \sigma^{\triangle}) = hd(\pi,\sigma) - 3$, it remains only to show that $\pi(j) = b$ and $\sigma(i) = a$. For simplicity we let $r = s = 1$, since the argument 
does not depend on $r = s$.  Solving for $b$ in $(b-a)(j-i) = 1$ we get $b = \frac{1}{j-i} + a = \pi(j)$.  Solving for $a$ we 
get $a = b - \frac{1}{j-i} = b + \frac{1}{i - j} = \sigma(i)$, as required.       \end{proof}

\medskip

\begin{lemma} \label{associates2}  Let $q = p^{m}$, where $p$ is an odd prime, with $q\equiv 1($mod $3)$, $q\geq 13$.  
Let $\pi, \sigma\in P$, say with $\pi(x) = a + \frac{r}{x-i}$, $\sigma(x) = b + \frac{s}{x-j}$, with $r,s\ne 0$.  Then 
$hd(\pi^{\triangle}, \sigma^{\triangle}) = hd(\pi,\sigma) - 3 \Longleftrightarrow \pi\sigma\in E(C_{P}(q))$.

 \end{lemma} 
 
 \begin{proof}$\Longleftarrow$: By definition of edges in $C_{P}(q)$ and Lemma \ref{basicsintro}a we have $q - 4 = hd(\pi^{\triangle}, \sigma^{\triangle})\geq hd(\pi, \sigma) -3$.    
 Now since $q-1\le hd(\pi, \sigma)\le q+1$, equality is forced together with $hd(\pi, \sigma) = q-1$.  This yields $hd(\pi^{\triangle}, \sigma^{\triangle}) = hd(\pi,\sigma) - 3$.   
 
 \noindent $\Longrightarrow$ : By the assumption $hd(\pi^{\triangle}, \sigma^{\triangle}) = hd(\pi,\sigma) - 3$ and $hd(\pi, \sigma)\geq q-1$ we are reduced to showing that 
 $hd(\pi, \sigma) = q-1$; that is, that $\pi$ and $\sigma$ already agree in two coordinates.
 
 By assumption and Lemma \ref{associates1} we have $r=s$, so write
 $\pi(x) = a + \frac{r}{x-i}$ and $\sigma(x) = b + \frac{r}{x-j}$, for $a, b, i, j, k \in GF(q)$ with $r\ne 0$.  Note also 
 $i\ne j$, since otherwise by Lemma \ref{associates1} we get $r=0$, a contradiction.

 We now derive a quadratic equation over $GF(q)$ whose distinct roots are the coordinates of agreement between $\pi$ and $\sigma$.
 Since $hd(\pi^{\triangle}, \sigma^{\triangle}) = hd(\pi,\sigma) - 3$, by Lemma \ref{associates1} 
 we have $(b - a)(j - i) = r$.  Thus $b = \frac{r}{j - i} + a$.  Now we 
 set $\pi(x) = \sigma(x)$ to find the possible coordinates $x$ at which $\pi$  and $\sigma$ agree, understanding that $x$ can be neither $i$ nor $j$ 
since $\pi$ and $\sigma$ can have no agreements in any of the coordinates $i = \pi^{-1}(\infty),j = \sigma^{-1}(\infty)$, or $\infty$ by Lemma \ref{basicsintro}a.
Substituting $\frac{r}{j - i} + a$ for $b$ and simplifying 
we obtain $\frac{1}{x-i} - \frac{1}{x-j} = \frac{1}{j-i}.$  Hence $\frac{i - j}{(x - i)(x - j)} = \frac{1}{j-i}$, and we get 
the quadratic $x^{2} - (i + j)x + ij + (i - j)^{2} = 0.$  By Corollary \ref{quadroots}b there are two 
 distant roots to this equation, giving the two coordinates of agreement for $\pi$ and $\sigma$ as follows; 
  $x_{1} =  \frac{1}{2}[i(1+\sqrt{-3}) + j(1-\sqrt{-3})]$, and $x_{2} =  \frac{1}{2}[i(1-\sqrt{-3}) + j(1+\sqrt{-3})]$.  
  
 Hence by our reduction at the beginning of the proof it follows that $\pi\sigma\in E(C_{P}(q))$, as required.        \end{proof}

 The preceding two Lemmas yield the following.

\begin{corollary}\label{edges fields} Let $q = p^{m}$, where $p$ is an odd prime, with $q\equiv 1($mod $3)$, $q\geq 13$.

\smallskip
  
\noindent {\textbf a)} Let $\pi, \sigma\in P$, say with $\pi(x) = a + \frac{r}{x-i}$, $\sigma(x) = b + \frac{s}{x-j}$, $r,s\ne 0$. Then
$\pi\sigma\in E(C_{P}(q)) \Longleftrightarrow r = s$ and $(b-a)(j-i) = r$.

\medskip

\noindent {\textbf b)} $\pi\in PGL(2,q)$ is an isolated point in 
$C_{P}(q) \Longleftrightarrow \pi(\infty) = \infty$.

\end{corollary}

\begin{proof} Part a) follows immediately from Lemmas \ref{associates1} and \ref{associates2}.

For part b), suppose first that $\pi(\infty) = \infty$.  Then immediately $\pi$ is isolated in $C_{P}(q)$ by the equivalence (\ref{contractionedges}) (with $\infty = F$)  applicable to 
any contraction graph. 

Conversely, suppose to the contrary that $\pi$ is isolated in $C_{P}(q)$ and $\pi(\infty) = x\ne \infty$.  Let $i = \pi^{-1}(\infty)$, and let $j$ be any coordinate with $j\notin \{ i, \infty \}$, 
and let $\pi(j) = y$.  Then by sharp $3$-transitivity of $PGL(2,q)$ we can find an element $\sigma\in PGL(2,q)$ satisfying $\sigma(j) = \infty$, $\sigma(i) = x$, 
and $\sigma(\infty) = y$.  Then we get $hd(\sigma^{\triangle}, \pi^{\triangle}) =  hd(\sigma, \pi) - 3$.  So by Lemma \ref{associates2} we have 
$\pi\sigma\in E(C_{P}(q))$, contradicting $\pi$ being isolated.      \end{proof}

The next two theorems, which use the preceding Corollary, tell us more about $C_{P}(q)$.  For $S\subset C_{P}(q)$, recall that $[S]$ is the subgraph 
of $C_{P}(q)$ induced by $S$.  
When $r$ is fixed by context, we denote a vertex $\pi\in C_{P}(q)$, $\pi\in P$, with $\pi(x) = a + \frac{r}{x-i}$, by the abbreviation $(i,a)$. 

Consider the partition of $P$ given by $P = \cup_{r\ne 0}P_{r}$, where for 
  $r\in GF(q)$ with $r\ne 0$, 
 $P_{r} = \{ a + \frac{r}{x-i} : a, i \in GF(q) \}$, so $|P_{r}| = q^{2}$.  Further consider the partition of $P_{r}$ given by $P_{r} = \cup_{i\in GF(q)}B_{i}(r)$, 
 where $B_{i}(r) = \{ a + \frac{r}{x-i} : a \in GF(q) \}$.

\begin{theorem}\label{first structure} Let $q = p^{m}$, where $p$ is an odd prime, with $q\equiv 1($mod $3)$, $q\geq 13$.  Then the following hold in the graph $C_{P}(q)$.

\noindent {\textbf a)} For any $r\ne s$, $r,s\ne 0$, we have $[P_{r}] \cong [P_{s}]$.

\noindent {\textbf b)} For any $r\ne 0$ and $i\ne j$, $[B_{i}(r)\cup B_{j}(r)]$ is a perfect matching, which matches $B_{i}(r)$ to $B_{j}(r)$. 

\noindent {\textbf c)} For any $r\ne 0$, the subgraph $[P_{r}]$ is regular of degree $q-1$.

\noindent {\textbf d)} Let $v\in C_{P}(q)$ be a non isolated point, and $N(v)$ the set of neighbors of $v$ in $C_{P}(q)$.  Then 
$[N(v)]$ is a disjoint union of cycles.

\end{theorem}

\begin{proof} For a), consider for any $r\in GF(q)$, $r\ne 0$, the map $\varphi : P_{1}\rightarrow P_{r}$ given by 
$\varphi(a + \frac{1}{x-i}) = a + \frac{r}{x-ri}$.  Let $v, w\in P_{1}$, say with $v(x) = a + \frac{1}{x-i}$ and $w(x) = b + \frac{1}{x-j}$.  Then $vw\in E([P_{1}])\Leftrightarrow (b-a)(j-i) = 1\Leftrightarrow (b-a)(rj-ri) = r
\Leftrightarrow \varphi(v)\varphi(w)\in E([P_{r}])$ .  Thus $\varphi$ is a graph isomorphism, and since $r$ was arbitrary, it follows that for any $s\ne 0$ we have $[P_{r}]\cong [P_{1}]\cong [P_{s}]$.

Consider b).  Fix $r$, and consider any two points $(i,a)$ and $(j,b)$ of $P_{r}$.  By Corollary \ref{edges fields} we have $(i,a)(j,b)\in E(C_{P}(q))$ if and only if $i\ne j$ and 
$(b-a)(j-i) = r$ in $GF(q)$.  Let $H_{ij} = [B_{i}(r)\cup B_{j}(r)]$ for $i\ne j$.  Note there can be no edge in $H_{ij}$ of the form $(i,a)(i,b)$ since $(b-a)(i-i) = 0\ne r$, and similarly no edge of the form $(j,a)(j,b)$.  
Now given $(i,a)\in B_{i}(r)$, a point $(j,b)\in B_{j}(r)$ is a neighbor of $(i,a)$ if and only if $(b-a)(j-i) = r$ by Corollary \ref{edges fields}.

Thus for this fixed $i$ and $j$ we can uniquely determine $b$ by the equation $b = r(j-i)^{-1} + a$, showing that $(j,b)$ is the only neighbor of 
$(i,a)$ in $B_{j}(r)$.  A symmetric argument shows that each point in $B_{j}(r)$ has a unique neighbor in $B_{i}(r)$.  Thus $E(H_{ij})$ is a perfect matching, which matches $B_{i}(r)$ to $B_{j}(r)$.

\begin{figure}[htb]
\begin{center}
\includegraphics{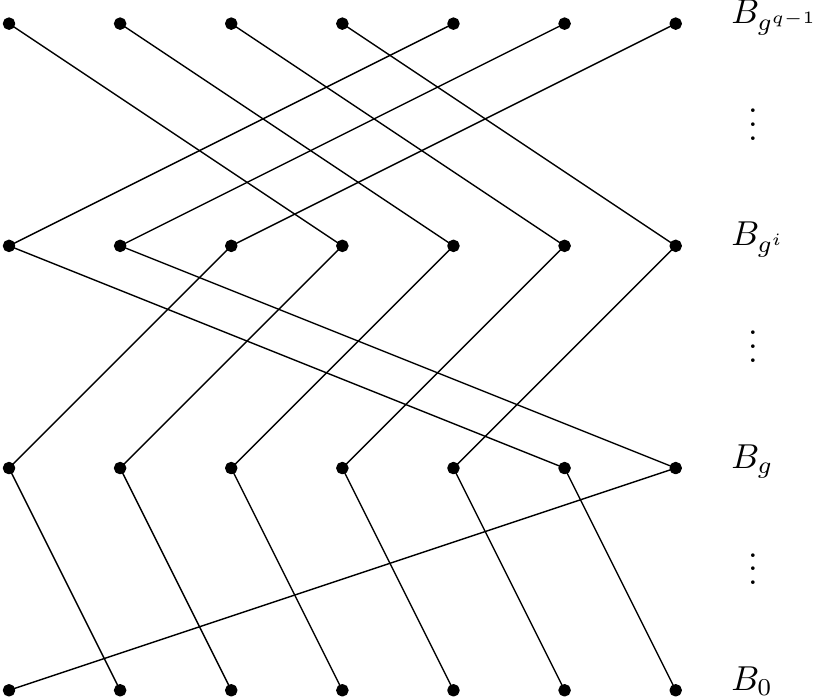}
\end{center}
\caption{Perfect matching between any two levels of $P_1$.}
\label{f2}
\end{figure}

For c), let $v\in C_{P}(q)$, say with $v\in B_{i}(r)\subset P_{r}$ for some $r\ne 0$.   By Corollary \ref{edges fields}, any neighbor of $v$ in $C_{P}(q)$ must also lie in $P_{r}$.  By part b), the neighbors of $v$ are in one to one 
correspondence with the sets $B_{j}(r)$, $j\ne i$, $j\in GF(q)$.  Thus $v$ has exactly $|GF(q)|-1 = q-1$ neighbors in $C_{P}(q)$. 

For d), take $v\in C_{P}(q)$, and by the isomorphism of subgraphs $[P_{r}]$ from part a), 
we can take $v = (i,a)\in P_{1}$.  By Corollary \ref{edges fields} we have $N(v)\subset P_{1}$.  It suffices to show that 
$[N(v)]$ is regular of degree 2.  Let $(j,b)\in N(v)$, so $j\ne i$ by part b).  Now any neighbor $(k,c)$ of $(j,b)$ in $N(v)$ must lie in 
$N((i,a))\cap N((j,b))$.  So to show that $(j,b)$ has degree $2$ in $[N(v)]$, 
it suffices to show that $(k,c)\in P_{1}$ satisfies $(k,c)\in N((i,a))\cap N((j,b))$ if and only if $k$ 
is a root in $GF(q)$ of a quadratic equation over $GF(q)$ having two distinct roots in $GF(q)$.

Suppose first that $(k,c)\in N((i,a))\cap N((j,b))$.  By Corollary \ref{edges fields} we must have the equations 
$$(c-a)(k-i) = 1, (b-c)(j-k) = 1, (b-a)(j-i) = 1.$$  
Using the second and third equations we get $c = (j-i)^{-1} - (j-k)^{-1} + a$, and from the first equation $c = (k-i)^{-1} + a$.  
Setting these 
two expressions for $c$ equal we obtain $(k-i)^{-1} + (j-k)^{-1} = (j-i)^{-1}.$  Some simplification leads to the 
quadratic $k^{2} - k(i+j) + ij +(j-i)^{2} = 0$ with coefficients over $GF(q)$ and unknown $k$. By Corollary \ref{quadroots}b from the Appendix,  
we see that that there are two distinct solutions for $k$; namely 
$k_{1} =  \frac{1}{2}[i(1+\sqrt{-3}) + j(1-\sqrt{-3})]$, and $k_{2} =  \frac{1}{2}[i(1-\sqrt{-3}) + j(1+\sqrt{-3})].$   

Conversely suppose that $k$ is one of the two distinct solutions of $k^{2} - k(i+j) + ij +(j-i)^{2} = 0$.  Then 
$(k-i)(j-k) = -k^{2} + k(i+j) - ij = (j-i)^{2}$, and using $\frac{1}{(k-i)(j-k)} = \big( \frac{1}{j-i} \big)\big( \frac{1}{k-i} + \frac{1}{j-k} \big)$, one can derive 
$\frac{1}{k-i} + \frac{1}{j-k} = \frac{1}{j-i}$. 
Now set $c = \frac{1}{k-i} + a$, so immediately we get $(c-a)(k-i) = 1$.  Since $(i,a)$ and $(j,b)$ are neighbors we have $(b-a)(j-i) = 1$, 
so $b = \frac{1}{j-i} + a$.  It follows that $c = \frac{1}{k-i} + a = \frac{1}{j-i} - \frac{1}{j-k} + a = b - \frac{1}{j-k}$.  Hence we get $(b-c)(j-k) = 1$.  Thus the three equations  
$(c-a)(k-i) = 1$, $(b-c)(j-k) = 1$, and $(b-a)(j-i) = 1$ hold, showing that $(k,c)\in N((i,a))\cap N((j,b))$ by Corollary \ref{edges fields}. 

Note that once $k$ is determined (as one of the two distinct roots), 
then the point $(k,c)$ is uniquely 
determined by the perfect matching between $B_{k}(1)$ and $B_{i}(1)$ (or $B_{j}(1)$).  Thus we obtain that an arbitrary point $(j,b)\in N(v)$ has exactly two neighbors in $N(v)$, 
completing d).

\end{proof}

To round out the structure of $C_{P}(q)$ we consider the connected components of $C_{P}(q)$.  

\begin{theorem}  Let $q = p^{m}$, where $p$ is an odd prime, with $q\equiv 1($mod $3)$, $q\geq 13$.   Then the connected components of $C_{P}(q)$ are as follows.

\noindent 1) the isolated points - these are of the form $\pi(x) = ax + b$, $a\ne 0 $, and there are $q(q-1)$ of them,

\noindent 2) the $q-1$ many connected components $[P_{r}]$ induced by the sets $P_{r}$.

\end{theorem}

\begin{proof} By Corollary \ref{edges fields}b we have that $\pi\in PGL(2,q)$ is an isolated point in $C_{P}(q)$ if and only if $\pi(\infty) = \infty$. 
This is equivalent to  $\pi(x) = ax + b$, $a\ne 0$ and there are $q(q-1)$ such points, completing part 1).

The remaining permutations are all of the form $\pi(x) = a + \frac{r}{x-i}$ for suitable $a,r,i\in GF(q)$ with $r\ne 0$ as shown earlier.  
Hence it suffices to analyze the connected component structure 
of $[\cup_{r\ne 0} P_{r}]$.  By Corollary \ref{edges fields} and Theorem \ref{first structure}a, 
to prove part 2) it suffices to prove that any one of the $[P_{r}]$, say $[P_{1}]$, is connected.

Recall the partition $P_{1} = \cup_{i\in GF(q)}B_{i}(1)$ defined above, and from now on we abbreviate $B_{i}(1)$ by $B_{i}$.  
Let $g$ by a generator of the multiplicative cyclic subgroup of nonzero elements in $GF(q)$.
Then we can write this partition as $P_{1} = B_{0}\cup (\cup_{1\le k\le q-1}B_{g^{k}})$.
We regard the sets in this partition as ``levels" of $C_{P}(q)$; where $B_{0}$ is level $0$ and $B_{g^{k}}$ is level $k$, $1\le k\le q-1$.  
See Figure 2 for an illustration of $P_{1}$ from this viewpoint, where in that Figure we continue with the notation $(i,a)$ for $a + \frac{1}{x-i}$. In particular, 
$(g^{t}, a)$ refers to $a + \frac{1}{x-g^{t}}$.    
By Theorem \ref{first structure}b the subgraph of $[P_{1}]$ 
induced by any two levels has edge set which is a perfect matching, as illustrated in Figure 3.

First we observe that to show that $[P_{1}]$ is connected it suffices to show that any two vertices in $B_{0}$ are joined by a path in $[P_{1}]$.  For if that 
was true, then we can find a path in $[P_{1}]$ from $(0,0)$ to any vertex $w\in P_{1}$ (thus showing connectedness of $[P_{1}]$) as follows.  If $w\in B_{0}$ we are done by assumption.  So suppose $w\notin B_{0}$, 
say with $w\in B(g^{k})$.  Let $v$ be 
the unique neighbor in $B_{0}$ of $w$ under the perfect matching $E([B_{0}\cup B(g^{k})])$.  
Let $P$ be the path from $(0,0)$ to $v$ in $[P_{1}]$ which exists by assumption.  Then $P$ followed by the edge $vw$ is a walk joining $(0,0)$  
to $w$, so $P$ contains a path from $(0,0)$ to $w$.

By Theorem \ref{first structure}b there is a (unique) path in $[P_{1}]$ starting at $(0,0)$ and passing through levels $1,2,\cdots, q-1$ in succession. 
Let $(0,0) - (g, \alpha_{1}) - (g^{2}, \alpha_{2}) - ... - (g^{q-1}, \alpha_{q-1})$ be this path, illustrated in bold lines in Figure 4, for suitable $\alpha_{k}\in GF(q)$. 
For $k\geq 1$ let $(0, \beta_{k})\in B_{0}$ be the unique neighbor in level $0$ of the vertex $(g^{k}, \alpha_{k})$ in level $k$.   The edges $(g^{k}, \alpha_{k})(0, \beta_{k})$ 
are illustrated by the dotted lines in in Figure 4.

This path and the 
points $(0, \beta_{k})$ are illustrated in Figure 4.
Our first step is to obtain the values of 
$\alpha_{k}$ and $\beta_{k}$.

\begin{figure}[htb]
\begin{center}
\includegraphics{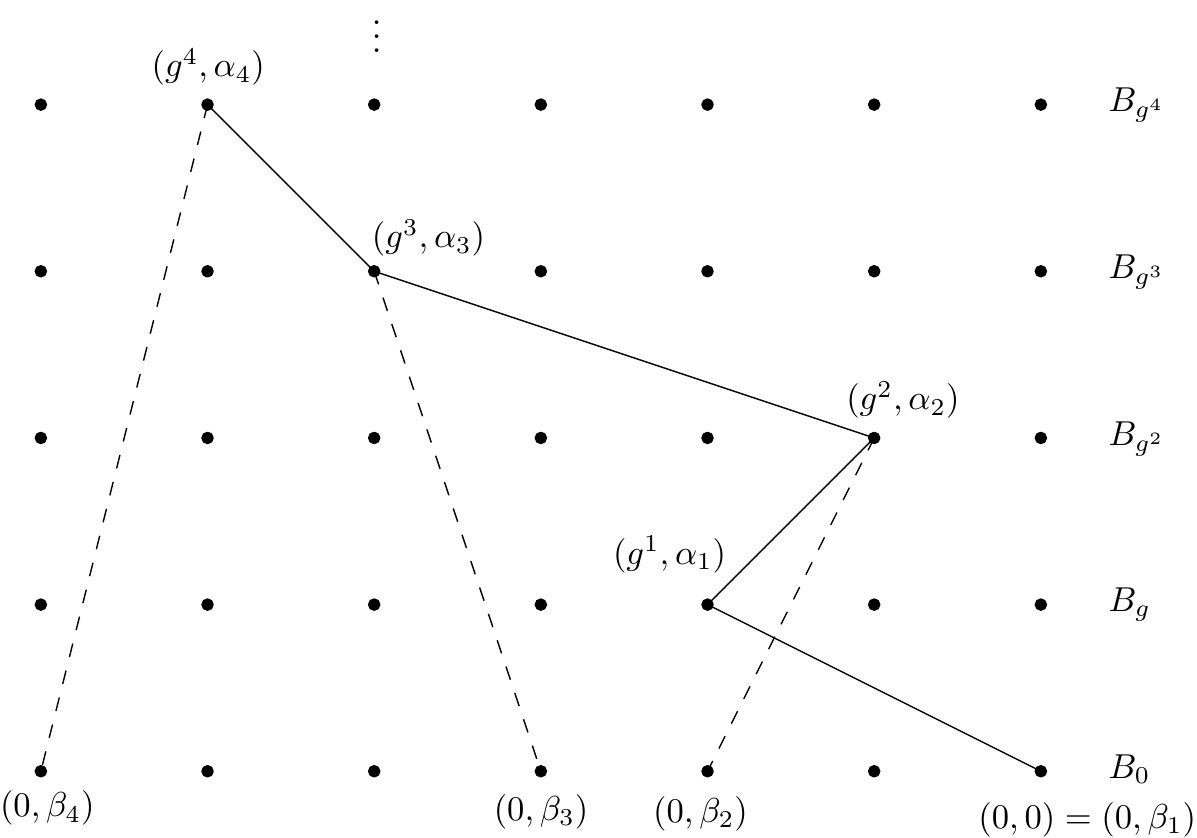}
\end{center}
\caption{The path $(0,0)-(g,\alpha_1)-(g^2,\alpha_2)-\dots -(g^{q-1},\alpha_{q-1})$ in $P_{1}$, where
$(0,\beta_i)$ is the level $0$ neighbor of $(g^i,\alpha_i)$.}
\label{f3}
\end{figure}

\smallskip

\noindent \underline{\textbf{Claim 1}:} We have
 
\noindent a) $\alpha_{1} = \frac{1}{g}, \alpha_{2} =  \frac{1}{g-1}$, and $\alpha_{k} = \frac{g^{k-1}+g^{k-3}+g^{k-4}+\cdots + g+1}{(g-1)g^{k-1}}$ for $k\geq 3.$ 

\noindent b) $\beta_{1} = 0$, and $\beta_{k} = \frac{(g^{2}-g+1)(1+g+g^{2}+g^{3}+\cdots + g^{k-2})}{g^{k}(g-1)}$ for $k\geq 2$.
\smallskip

\noindent \underline{Proof of Claim 1}: We repeatedly use the fact, proved earlier, that if $(r, a)$ and $(s, b)$ are adjacent vertices in the contraction graph $C_{P}(q)$, then 
$(s-r)(b-a) = 1$. 

For part a), since $(0,0) - (g, \alpha_{1})$ is an edge in $C_{P}(q)$ we have $(\alpha_{1}-0)(g-0) = 1$, so $\alpha_{1} = \frac{1}{g}.$  
Since $(g, \alpha_{1}) - (g^{2}, \alpha_{2})$ is an edge  we have $(\alpha_{2}-\frac{1}{g})(g^{2}-g) = 1$, 
yielding $\alpha_{2} = \frac{1}{g-1}$, and similarly $(\alpha_{3}-\frac{1}{g-1})(g^{3}-g^{2}) = 1$, yielding $\alpha_{3} = \frac{g^{2}+1}{(g-1)g^{2}}$.  Now for $k\geq 3$ we proceed by induction, having proved the base case 
$k = 3$.  Since $(g^{k}, \alpha_{k}) - (g^{k-1}, \alpha_{k-1})$ is an edge, 
we have $(\alpha_{k} - \alpha_{k-1})(g^{k} - g^{k-1}) = 1$.  Solving for $\alpha_{k}$ and applying 
the inductive hypothesis to $\alpha_{k-1}$, we obtain $\alpha_{k} = \frac{1}{g^{k}-g^{k-1}} + \frac{g^{k-2}+g^{k-4}+g^{k-5}+\cdots + g+1}{(g-1)g^{k-2}}$, which after simplification yields the claim.

For part b), we have $\beta_{1} = 0$ since $(0,0) - (g, \alpha_{1})$ is an edge by definition.  Since $(g^{2}, \alpha_{2}) - (0, \beta_{2})$ 
is an edge, we have $(\frac{1}{g-1} - \beta_{2})(g^{2}-0) = 1$, and solving for $\beta_{2}$ and simplifying we get the claim for $k=2$.  Consider now 
$k\geq 2.$  The existence of edge $(g^{k}, \alpha_{k}) - (0, \beta_{k})$ gives $({\alpha_{k} - \beta_{k}})g^{k} = 1$, so $\beta_{k} = \alpha_{k} - \frac{1}{g^{k}}$. 
Using the formula for $\alpha_{k}$ from part a), we have 
$\beta_{k} = \frac{g^{k-1}+g^{k-3}+g^{k-4}+\cdots + g+1}{(g-1)g^{k-1}} - \frac{1}{g^{k}} = \frac{g^{k}+g^{k-2}+g^{k-3}+ \cdots+g^{2}+1}{(g-1)g^{k}} = \frac{(g^{2}-g+1)(1+g+g^{2}+g^{3}+\cdots + g^{k-2})}{g^{k}(g-1)}$. 
QED

\medskip

\noindent \underline{\textbf{Claim 2}:} We have $|\{ \beta_{k}: 1\leq k\leq q-1\}| = q-1$; that is, the $\beta_{k}$, $1\le k\le q-1$, are pairwise distinct.

\smallskip

\noindent \underline{Proof of Claim 2}:   In applying Claim 1, we note first that $g$ could have been chosen so as not to be a root 
of $x^{2} - x + 1 = 0$ as follows.  The number of roots in $GF(q)$ to this quadratic is at most $2$.  Now the number of generators in the multiplicative cyclic group $GF(q) - \{0\}$ of order $q-1$ 
is the euler totient function $\phi(q-1)$, defined as the number of integers $1\le s\le q-1$ which are relatively prime to $q-1$.  Since $q$ is an odd prime power with $q\geq 13$, we know that $\phi(q-1) > 2$, so 
such a $g$ exists.

We show that for for any pair $j,k$ with $1\le j < k \le q-1$ we have $\beta_{k}\ne \beta_{j}$.  
 
 Consider first the case $j = 1$.  Since $\beta_{1} = 0$, we need to show that $\beta_{k}\ne 0$ for $2\le k\le q-1$.  Supposing the 
contrary and applying Claim 1b we get
 $\frac{(g^{2}-g+1)(1+g+g^{2}+g^{3}+\cdots + g^{k-2})}{g^{k}(g-1)} = 0.$  Canceling the nonzero factor $\frac{g^{2}-g+1}{g^{k}(g-1)}$ (by the preceding paragraph) 
 on the left side, we get $0 = (1+g+g^{2}+g^{3}+\cdots + g^{k-2}) = \frac{g^{k-1} - 1}{g - 1}.$  
 This implies that $g^{k-1} - 1 = 0$, so $g$ has order $k-1$.  This is impossible since $k-1\le q-2$ while $g$, being a generator of the cyclic group $GF(q) - \{0\}$, 
 must have order $q-1$.
 
 So now suppose that $j\geq 2$.  Assuming the contrary that $\beta_{k} = \beta_{j}$ and applying Claim 1b, we get after simplification that 
 $1+g+g^{2}+g^{3}+\cdots + g^{k-2} = g^{k-j}(1+g+g^{2}+g^{3}+\cdots + g^{j-2}) = g^{k-j} + g^{k-j+1} + \cdots + g^{k-2}$.   Thus we have 
 $0 = 1 + g + g^{2} + \cdots + g^{k-j-1} = \frac{g^{k-j} - 1}{g - 1}$.  So $g^{k-j} = 1$, which is impossible since $k - j\le q - 3$, while $g$ has order $q - 1$. QED
 \medskip        
                      
We introduce some notation in preparation for the rest of the argument.  Let $Z = \{ (0,\beta_{k}): 1\leq k\leq q-1\}\subset B_{0}$.  
Since $|B_{0}| = q$, by Claim 2 we have $|B_{0}  - Z| = 1$, and we let $u$ be the unique vertex of $B_{0} - Z$.  Further 
for any subset $T$ of vertices in $C_{P}(q)$, we let $N(T) = \{ v\in C_{P}(q): v\notin T, vt\in E(C_{P}(q))$ for some $t\in T\} $  
be the neighbor set of $T$ in $C_{P}(q)$.  Recall also that $[T]$ denotes the subgraph of $C_{P}(q)$ induced by $T$.
\smallskip

\noindent \underline{\textbf{Claim 3}:} Let $H = [Z\cup N(Z)]_{C_{P}(q)}$, and $H' = [\{u\}\cup N(u)]_{C_{P}(q)}$. 

\noindent a) $H'$ is connected.

\noindent b) $H$ is connected.

\noindent c) $V(H)\cap V(H') = \emptyset$ 

\noindent d) We have the partition $V(P_{1}) = V(H)\cup V(H')$.

\smallskip 

\noindent \underline{Proof of Claim 3}: For part a), we apply Theorem \ref{first structure}b to deduce that $H'$ has the spanning star  subgraph $K_{1, q-1}$, where 
the center is $u$ and the leaves, one in each level $B_{i}$, $i\ne 0$, form $N(u)$.  Thus  $H'$ is connected.

Consider part b).  Since $\beta_{1} = 0$ we have $(0,0)\in Z\subset V(H)$. 
Thus it suffices to show that for any $w\in V(H)$ there is a path in $H$ joining $(0,0)$ to $w$.  

Suppose first that $w\in Z$, so $w = (0, \beta_{k})$ for some $k$.  Observe that $(g^{i},\alpha_{i}) \in N(Z)$ for all $i$ by definition.  So the path $(0,0) - (g, \alpha_{1}) - (g^{2}, \alpha_{2}) - ... - (g^{k}, \alpha_{k})$ 
followed by the edge $(g^{k}, \alpha_{k}) - (0,\beta_{k})$ is path in $H$ joining $(0,0)$ to $w$.  

Next suppose $w\in N(Z)$, say with $w$ adjacent to $(0, \beta_{k})\in Z$.  
Then the path $(0,0) - (g, \alpha_{1}) - (g^{2}, \alpha_{2}) - ... - (g^{k}, \alpha_{k})$ followed by the length $2$ path 
$(g^{k}, \alpha_{k}) - (0, \beta_{k}) - w$ is a walk in $H$ joining $(0,0)$ to $w$, and this walk  contains the required path. 

Next consider c).  Suppose not, and let $z\in V(H)\cap V(H')$, say with $z\in B(g^{k})$, noting that $k\geq 1$ 
since each level, in particular $B_{0}$, is an independent set in $[P_{1}]$.  
Then $z$ has two distinct neighbors in $B_{0}$; namely $u$ and $(0,\beta_{j})$, for 
some $1\le j\le q-1$.  This contradicts the fact that the edge set of $[B(g^{k})\cup B_{0}]$ is 
a perfect matching between the levels $B(g^{k})$ and $B_{0}$ by Theorem \ref{first structure}b.  Thus $V(H)\cap V(H') = \emptyset$.   

Consider now d).  By part c), it suffices to show that $|V(P_{1})| = |V(H)| + |V(H)'|$. 
By Theorem \ref{first structure}b, it follows that $|V(H)| = |Z|q = (q-1)q.$   For the same reason $|V(H')| = q.$  Therefore 
$|V(P_{1})| = q^{2} = |V(H)| + |V(H)'|$ as required.  QED   

\medskip

We can now complete the proof of the theorem by showing that $P_{1}$ is connected.  In view of Claim 3, to do this we are
reduced to showing that there is an edge $vw\in E([P_{1}])$ with $v\in H'$ and $w\in H$.  
Suppose no such edge exists.  Since $[P_{1}]$  is $(q-1)$-regular by Theorem \ref{first structure}c, 
it follows that $H'$ is a simple $q-1$ regular graph on $q$ vertices.  Thus  
 $H' = K_{q}$. Hence $[N(u)] =  K_{q-1}$.  But this is a contradiction for $q\geq 5$ since by 
 Theorem \ref{first structure}d the neighborhood of any nonisolated point in $C_{P}(q)$ is regular of degree $2$, 
 while $[N(u)] $ is regular of degree $q-2 > 2$ since we have assumed $q\geq 13$.     \end{proof}

We can now obtain our independent set in $C_{P}(q)$ as a consequence of our previous results and the following theorem of Alon \cite{Alon}.

\begin{theorem} \cite{Alon} Let $G$ = (V,E) be a graph on $N$ vertices with average degree $t\geq 1$ in which for every vertex $v\in V$ the induced subgraph on the set 
of all neighbors of $v$ is $r$-colorable.  Then the maximum size $\alpha(G)$ of an independent set in $G$ satisfies 
$\alpha(G)\geq \frac{c}{\log(r+1)}\frac{N}{t}\,\log t$, for some absolute constant $c$. 
\end{theorem}

\begin{corollary} Let $q$ be a power of an odd prime $p$, with $q\equiv 1($mod $3)$, $q\geq 13$.  

\noindent {\textbf a)} $\alpha(C_{P}(q))\geq Kq^{2}\log q$ for some constant $K$.

\noindent {\textbf b)} $M(q, q-3) \geq Kq^{2}\log q$ for some constant $K$.

\end{corollary}

\begin{proof} Consider a).  By Corollary \ref{edges fields}a there is no edge between any two subgraphs $[P_{r}]$ and $[P_{s}]$ for $r\ne s$.  Since there are $q$ such subgraphs, and by Theorem \ref{first structure}a) they are 
pairwise isomorphic, it suffices to show that $\alpha(P_{1}) \geq Kq\log q$ for some constant $K$.

We now apply Alon's theorem to the subgraph $[P_{1}]$ of $C_{P}(q)$.  
Now  $[P_{1}]$ is $(q-1)$-regular by Theorem \ref{first structure}c, and has $q^{2}$ points.  
Since the neighborhood of every point is 
a disjoint union of cycles by Theorem \ref{first structure}d, this neighborhood must be $3$-colorable.  
It follows by Alon's theorem that $[P_{1}]$ contains an independent set of size $\frac{c}{\log 4}\frac{q^{2}}{q-1}\,\log(q-1) \sim Kq\log q$, for some constant $K$.

For b), let $I$ be an independent set in $C_{P}(q)$ 
of size $Kq^{2}\log q$ for suitable constant $K$, guaranteed to exist by by part a).
Then by the reduction made in the discussion preceding 
Lemma \ref{associates1} we have $M(q,q-3) \geq |I| \geq Kq^{2}\log q$.     
\end{proof}

\section{Special case lower bounds for $M(n,d)$ via the Mathieu groups}

In this section we consider the Mathieu groups $M_{11}, M_{12}, M_{22}, M_{23}, M_{24}$, discovered by E. Mathieu in 1861 and 1873.  
These permutation groups are the earliest known example of sporadic simple groups.  See \cite{DM}, \cite{C}, or \cite{Tas} for a discussion of their 
construction.  These groups act on $11, 12, 22, 23, 24$ letters respectively, with $M_{11}$ being a $1$ point stabilizer of $M_{12}$, while $M_{23}$ and $M_{22}$ are $1$ 
and $2$ point stabilizers of $M_{24}$ respectively.  

In this section we apply the contraction operation to these permutation groups to obtain new permutation arrays, with resulting lower bounds for $M(n,d)$ for suitable $n$ and $d$.  

Since $M_{12}$ is sharply $5$-transitive we have by Theorem \ref{Deza} that $hd(M_{12}) = 8$ and $M(12,8) = |M_{12}| = 95040$.  Similarly since $M_{11}$ 
is sharply $4$-transitive we have $M(11,8) = |M_{11}| = 7920$.  For 
$M_{24}$ we do not have sharp transitivity.  But observe that for any permutation group $G$ acting on some set, and three elements $\pi, \sigma, \tau\in G$, we have 
$hd(\pi, \sigma) = hd(\pi\tau, \sigma\tau) = hd(\tau\pi, \tau\sigma)$.  Thus $hd(G) =$min$\{hd(1,\sigma): \sigma\in G \}$.  From the set of disjoint cycle structures of elements of $M_{24}$ 
(available at \cite{WikiM24}) we find that the largest number of $1$-cycles in the disjoint cycle structure of any nonidentity element of $M_{24}$ is $8$.  Thus  $hd(M_{24}) = 24-8 = 16$, 
and from the 
stabilizer relation also $hd(M_{23}) = hd(M_{22}) = 16$.  
We thus obtain $M(24,16)\geq |M_{24}| = 24,423,040$, $M(23,16)\geq |M_{23}| = 10,200,960$, and 
$M(22,16)\geq |M_{22}| = 443,520$.

We now apply the contraction operation to these groups.  Considering the action of $M_{12}$ on the $12$-letter set $\Omega = \{x_{1}, x_{2}, \cdots, x_{12}\}$, we designate 
some element, say $x_{12}$, of $\Omega$ as the distinguished element $F$ in the definition of $\pi^{\triangle}$.  
Then define for each $\pi\in M_{12}$ the permutation $\pi^{\triangle}$ on the set $\Omega$ exactly as in the introduction. 
Thus, using the natural ordering of elements of $\Omega$ by subscript, 
the image string of any $\sigma\in M_{12}$ can be written $\sigma(x_{1})\sigma(x_{2})\cdots \sigma(x_{11})\sigma(F)$.

As before, we let $\pi^{\triangle}_{-}$ be the permutation on $11$ symbols obtained 
from $\pi^{\triangle}$ by dropping the final symbol $F$, and for any subset $S\subset M_{12}$, 
we let $S^{\triangle}_{-} = \{ \pi^{\triangle}_{-}: \pi\in S\}$, sometimes writing this as $(S)^{\triangle}_{-}$. 

\begin{proposition} \label{M12} 

\noindent {\textbf a)} $hd((M_{12})^{\triangle}_{-}) \geq 6$.

\noindent {\textbf b)} $M(11,6)\geq |M_{12}| = 95040$.

\noindent {\textbf c)} $M(10,6)\geq 8640$.

\end{proposition} 

\begin{proof} We start with a).  Suppose not.  Since $hd(M_{12}) = 8$, and for any $\alpha, \beta\in M_{12}$ we have 
$hd(\alpha^{\triangle}, \beta^{\triangle})\geq hd(\alpha, \beta) - 3$ by Lemma \ref{basicsintro}a, the contrary assumption implies $hd((M_{12}^{\triangle})_{-}) = 5$.  Thus there is a pair $\sigma, \tau\in M_{12}$ 
such that $hd(\sigma,\tau) = 8$ and $hd(\sigma^{\triangle},\tau^{\triangle}) = 5$; so $hd(\sigma^{\triangle},\tau^{\triangle}) = hd(\sigma,\tau) - 3$.  
Thus by Lemma \ref{basicsintro}b we know that $\pi\sigma^{-1}$ has a $3$-cycle in its disjoint cycle factorization so the order of $\pi\sigma^{-1}$
is divisible by $3$.  

Since $hd(\sigma,\tau) = 8$ and $\pi$ and $\sigma$ are permutations on 12 letters, it follows that there are four 
positions, call them $x_{i}$, $1\le i\le 4$, at which $\pi$ and $\sigma$ agree.  
Then $\pi\sigma^{-1}$ belongs to the subgroup $H$ of $M_{12}$ fixing these four positions; that is $H = \{ \alpha\in M_{12}: \alpha(x_{i}) = x_{i}, 1\le i\le 4\}$. 
This $H$, denoted $M_{8}$, is known to be isomorphic to $Q_{8}$, the quaternion group of order $8$ (\cite{Boy}, section 3.2).  
We can also verify this directly by making use of GAP (Groups, Algorithms, Programming), a system for computational discrete algebra.  The following output employing GAP 
shows that $H\cong Q_{8}$, the quaternion group of order $8$ (\cite{Holt})

\medskip

\noindent gap$>$ $G: = MathieuGroup(12)$;;

\noindent gap$>$ $H = Stabilizer(G,[1,2,3,4], OnTuples)$;;

\noindent gap$>$ $StructureDescription(H)$;

\noindent $``Q_{8}''$

\medskip

Now the order of $\pi\sigma^{-1}$ is divisible by $3$ as noted above.   But $3$ does not 
divide $|Q_{8}|$, a contradiction to Lagrange's theorem.

Consider next b).  Using Lemma \ref{basicsintro}c and $hd(M_{12}) = 8 > 3$, we have $|M_{12}| = |(M_{12})^{\triangle}_{-}|$.  
Thus $(M_{12})^{\triangle}_{-}$ is a permutation array on 11 letters of size $|M_{12}| $ with $hd((M_{12})^{\triangle}_{-}) \geq 6$.  Part b) follows.

For part c), we recall from the introduction the elementary bound $M(n-1,d)\geq \frac{M(n,d)}{n}$.  Using part a), we then obtain $M(10,6)\geq \frac{M(11,6)}{11}\geq 8640.$         \end{proof}   

We remark that using the same method as in part b) of the above proposition one can show $M(10,6)\geq |M_{11}| = 7920$.  But this is obviously weaker than the bound we give in part c).

We now consider the contraction of $M_{24}$ and resulting special case bounds for $M(n,d)$.  Using similar notation as for $M_{12}$ above, we let $M_{24}$ act on the set of 
$24$ letters $\Theta = \{x_{1}, x_{2}, \cdots, x_{24}\}$, and we designate $x_{24}$ as the distinguished symbol $F$ in the definition of $\pi^{\triangle}$ from the introduction.  
Now define $\pi^{\triangle}$ for any $\pi\in M_{24}$ as in the introduction, along with accompanying definitions $S^{\triangle}$ and $S^{\triangle}_{-}$ for $S\subseteq M_{24}$.

\bigskip

\begin{proposition} \label{M24} 

\noindent {\textbf a)} $hd((M_{24})^{\triangle}_{-}) \geq 14$.

\noindent {\textbf b)} $M(23,14)\geq |M_{24}| = 244,823,040$.

\noindent {\textbf c)} $M(22,14)\geq \frac{|M_{24}|}{23} = 10,644,480$.

\noindent {\textbf d)} $M(21,14)\geq \frac{|M_{24}|}{23^{.} 22} = 483,840$.

\end{proposition}

\begin{proof}

For a), suppose not.  Since $hd(M_{24}) = 16$, and for any $\alpha, \beta\in M_{24}$ we have 
$hd(\alpha^{\triangle}, \beta^{\triangle})\geq hd(\alpha, \beta) - 3$, it follows that $hd((M_{24})^{\triangle}_{-}) = 13$.  Thus there is pair $\sigma, \tau\in M_{24}$ 
such that $hd(\sigma,\tau) = 16$ and $hd(\sigma^{\triangle},\tau^{\triangle}) = 13$; so $hd(\sigma^{\triangle},\tau^{\triangle}) = hd(\sigma,\tau) - 3$.  
Hence by Lemma \ref{basicsintro}b, $\tau\sigma^{-1}$ has a $3$-cycle in its disjoint cycle structure factorization.

Since $hd(\sigma,\tau) = 16$, and 
$\sigma$ and $\tau$ are permutations on $24$ letters, it follows that $\sigma$ and $\tau$ must agree on $8$ positions.  Thus 
$\tau\sigma^{-1}$ belongs to the subgroup $H$ of $M_{24}$ fixing these $8$ positions.  
From the structure theory of $M_{24}$, we know that if these $8$ positions form an ``octad" (among the $24$ positions), then 
$H = M_{16} \cong Z_{2}\times Z_{2}\times Z_{2}\times Z_{2}$, the elementary Abelian group of order 16 (\cite{Tas} Theorem 3.21, and \cite{Thom} pp. 197-208). Again, 
this can also be verified directly using GAP from the following output (\cite{Holt}).
\medskip

\noindent gap$>$ $G: = MathieuGroup(24)$;;

\noindent gap$>$ $H: = Stabilizer(G, [1,2,3,4,5], OnTuples)$;;

\noindent gap$>$ $S = SylowSubgroup(H,2)$;;

\noindent gap$>$ $octad: = Filtered([1..24],  x\rightarrow$ not $x$ in MovedPoints$(S) )$;

\noindent [1,2,3,4,5,8,11,13]

\noindent gap$>$ $H: = Stabilizer(G, octad, OnTuples)$;;

\noindent gap$>$ $StructureDescription(H)$;

\noindent $``C_{2}\times C_{2}\times C_{2}\times C_{2}''$.

\medskip

\noindent If these $8$ positions do not form an octad, then $H$ is the identity (\cite{Tas}, Lemma 3.1).  Now the order of $\tau\sigma^{-1}$ is divisible by $3$, so $3$ 
must divide $|H|$.  By Lagrange's theorem, this contradicts that $|H|$ has order either $16$ or $1$.

Consider next b).  Using Lemma \ref{basicsintro}c and $hd(M_{24}) = 16 > 3$, we have $|M_{24}| = |(M_{24})^{\triangle}_{-}|$.   
Thus $(M_{24})^{\triangle}_{-}$ is a permutation array on $23$ letters of size $|M_{24}| $, and by part a) 
we have $hd((M_{24})^{\triangle}_{-}) \geq 14$.  Part b) follows.

For part c), we again use the bound $M(n-1,d)\geq \frac{M(n,d)}{n}$.  Using part b), 
we then obtain $M(22,14)\geq \frac{M(23,14)}{23}\geq \frac{|M_{24}|}{23} = 10,644,480$.

For d), using $M(n-1,d)\geq \frac{M(n,d)}{n}$ again 
we get $M(21,14)\geq \frac{M(22,14)}{22}\geq \frac{|M_{24}|}{23^{.} 22} = 483,840$.       
\end{proof}

\section{Concluding Remarks}

We mention some problems left open from our work.

\smallskip

\noindent \textbf{1.} Recall that if $I$ is an independent set in $C_{P}(q)$, then $M(q, q-3)\geq |I|$.  To find a large such $I$, one can 
focus on any nontrivial connected component, say $P_{1}$, of $C_{P}(q)$.  If $P_{1}$ contains an independent set of size $k$, then by 
the isomorphism of the connected components $P_{i}$, $1\le i\le q-1$, we get an independent set of size $k(q-1) + q(q-1) = (q-1)(k+q)$ in $C_{P}(q)$, where $q(q-1)$ 
counts the number of isolated points in $C_{P}(q)$.  Our lower bound $M(q, q-3)\geq Kq^{2}\log q$ implies, again by the isomorphism of 
components, that $\alpha(P_{1})\geq Cq\log q$ (where $\alpha(G)$ is the maximum size of an independent set in a graph $G$), for some constant $C$.  
We therefore ask whether an improvement on this lower bound for $\alpha(P_{1})$ can be found.   

Now $V(P_{1})$  can be viewed as a rectangular array $\{(i,a): i, a\in GF(q)\}$ as in Figure 2, where we let $i$ be the row index, and $a$ the column index.  
By Corollary \ref{edges fields}a an independent set in $P_{1}$ is just a subset $S$ of this array with the property that for any two points
 $(i,a), (j,b)\in S$ we have $(b-a)(j-i)\ne 1$ in $GF(q)$.  Using the integer programming 
package GUROBI, we computed independent sets in $P_{1}$ of size $k$ for various $q$.  This $k$, together with the resulting 
lower bound $(q-1)(k+q)$ for $M(q,q-3)$ are presented in Table 1.  The primes $q = 41, 47, 53, 59, 71, 83, 89$, for example, are not included in this table since 
$q\not\equiv 1($mod $3)$, and hence $M(q,q-3) \geq (q+1)q(q-1)$,  an improvement over the lower bound obtained using GUROBI.

\medskip

\noindent \textbf{2.} We also ask for good upper bounds on $\alpha(P_{1})$. 

\begin{table}[ht] \label{values}
\begin{center} 
\begin{tabular}{ccc}\hline
 
$q$&$k$&$(q-1)(k+q)\le M(q,q-3))$\\
\hline
7 & 13 & 120 \\
13 & 33 & 552 \\
19 & 81 & 1800 \\
31 & 122 & 4590 \\
37 & 191 & 8208 \\
43 & 191 & 9828 \\
49 & 226 & 13200 \\
61 & 314 & 22500 \\
67 & 340 & 26862 \\
73 & 382 & 32760\\
$79$&$415$&$38532$\\
$97$&$535$&$60672$\\
103 & 598 & 71502\\
109 & 637 & 80568\\
$121$&$2613$&$328080$\\
127 & 768 & 112770 \\
139 & 867 & 138828\\
151 & 945 & 164400\\
$157$&$984$&$177996$\\
$163$&$1031$&$193428$\\
169 & 1069 & 207984 \\
181 & 1174 & 243900 \\
193 & 1262 & 279360 \\
199 & 1310 & 298782 \\
211 & 1403 & 338940 \\
223 & 1496 & 381618 \\
229 & 1565 & 409032 \\
241 & 1671 & 458880 \\
277 & 1956 & 616308 \\
283 & 2009 & 646344 \\
289 & 2045 & 672192 \\
307 & 2197 & 766224 \\
313 & 2272 & 806528 \\
331 & 2396 & 899910 \\
337 & 2462 & 940464 \\
343 & 2501 & 972648 \\
\hline
\end{tabular}
\caption{Independent set of size k in $P_{1}$ obtained by integer programming, and resulting lower bound $(q-1)(k+q)$ for $M(q,q-3)$, when $q\equiv 1$(mod $3$).} 
\end{center}
\end{table}

\section{Appendix - Some facts from Number Theory}

In this section we review some facts from number theory that were used in this paper.

 We start with some notation.  For an odd prime $p$ and integer $r\not\equiv 0 ($mod $p)$, define the Legendre symbol $(\frac{r}{p})$ to be $1$ (resp. -1)  if 
 $r$ is a quadratic residue (resp. nonresidue); that is a square (resp. nonsquare) mod $p$.  
 If $r\equiv 0 ($mod $p)$, then define $(\frac{r}{p}) = 0.$  A couple of simple facts 
 about this symbol  are these.
 
 \begin{lemma}\label{legendre}  For an odd prime $p$ and integers $r$ and $s$ we have the following.
 
 \noindent a) $(\frac{-1}{p}) = 1$ if $p\equiv 1$(mod $4)$, and $(\frac{-1}{p}) = -1$ if $p\equiv 3$(mod $4)$. 
 
 \noindent b) $(\frac{rs}{p}) = (\frac{r}{p})(\frac{s}{p})$.

 \end{lemma} 
 
 \begin{proof} For a), suppose $p\equiv 1$(mod $4)$. So write $p = 4k + 1$, and consider the multiplicative group of nonzero elements mod $p$, which has order $4k$ and is cyclic.  Let $x$ be a generator of this group.  
 Then note that in this group we have $1 = x^{4k} = (x^{2k})^{2}$, while also $(-1)^{2} = 1$ in this group.  Since the quadratic $z^{2} - 1 = 0$ has exactly two solutions $z = 1$ or $-1$ in $GF(q)$, and since 
 $x^{2k}\ne 1$ since $x$ is  a generator,  it 
 follows that $x^{2k} = -1$.  Thus -1 is a square mod $p$.

 If $p\equiv 3$(mod $4)$, then this cyclic group has order $4k+2$ for some integer $k$.  This time we have $1 = (x^{2k+1})^{2}$, so that by the same reasoning as above we have $x^{2k+1} = -1$.  
 This shows that -1 is not a square mod $p$, 
 since it is on odd power of the generator.
 
 Consider now b).  Just observe that the product $rs$ is a square mod $p$ if and only if both $r$ and $s$ 
 are squares mod $p$ or if both $r$ and $s$ are non-squares mod $p$.  Part b) then follows immediately.         
 
 \end{proof}

 We now recall the quadratic reciprocity law.
 
 \begin{theorem} (Gauss Quadratic Reciprocity Law) For odd primes $p$ and $q$ we have
 
 $$(\frac{p}{q})(\frac{q}{p}) = (-1)^{(\frac{p-1}{2})(\frac{q-1}{2})}.$$ 
 
 \end{theorem}
 
 There are lots of proof of quadratic reciprocity in the literature, so we omit the proof here.
 
 Now let's apply these facts to determining $(\frac{-3}{p})$ for odd primes p. 
 
 \begin{theorem}\label{reciprocity} Let $p > 3$  be an odd prime. Then 
 
 \noindent a) If $p\equiv 1$ (mod $6$), then  -3 is a quadratic residue mod $p$.
 
 \noindent b) If $p\equiv 5$ (mod $6$), then  -3 is a quadratic nonresidue mod $p$.

 \end{theorem} 
 
 \begin{proof}  By the lemma above we have $(\frac{-3}{p}) = (\frac{-1}{p})(\frac{3}{p})$, while by quadratic reciprocity we have 
 $(\frac{3}{p}) = (\frac{p}{3})(-1)^{\frac{p-1}{2}}$.  Thus 
 $$(\frac{-3}{p}) = (-1)^{\frac{p-1}{2}}(\frac{-1}{p})(\frac{p}{3}).$$ 
 
 The factors on the right depend on the residue classes 
 of $p$ mod $4$ and $p$ mod $3$.  Thus we consider the four cases defined by the combinations of these two possibilities, obtaining results that 
 initially depend on the residue class of $p$ mod $12$.
 
 \smallskip
 
\noindent \underline{\bf{case 1}}: $p\equiv 1($mod $4)$ and  $p\equiv 1($mod $3)$; equivalently $p\equiv 1($mod $12)$. 
 
 Now $p\equiv 1($mod $3)$ says that $(\frac{p}{3}) = 1$.  
 Also $p\equiv 1($mod $4)$ implies  $(-1)^{\frac{p-1}{2}} = 1$ and by Lemma \ref{legendre} also implies $(\frac{-1}{p}) = 1$.  So by the formula above we have $(\frac{-3}{p}) = 1$, showing that 
 $-3$ is a quadratic residue when $p\equiv 1($mod $12)$.
 
 \smallskip
 
\noindent \underline{\bf{case 2}}: $p\equiv 1($mod $4)$ and  $p\equiv 2($mod $3)$; equivalently $p\equiv 5($mod $12)$. 
  
  Now $p\equiv 2($mod $3)$ says that $(\frac{p}{3}) = -1$.  Also $p\equiv 1($mod $4)$ implies  $(-1)^{\frac{p-1}{2}} = 1$ and 
  also Lemma \ref{legendre} implies $(\frac{-1}{p}) = 1$.  
So by the formula above we have $(\frac{-3}{p}) = -1$, showing that 
 $-3$ is a quadratic nonresidue when $p\equiv 5($mod $12)$. 
 
 \smallskip
 
\noindent \underline{\bf{case 3}}: $p\equiv 3($mod $4)$ and  $p\equiv 1($mod $3)$; equivalently $p\equiv 7($mod $12)$.

Since $p\equiv 1($mod $3)$ we have $(\frac{p}{3}) = 1$.  Also $p\equiv 3($mod $4)$ implies  $(-1)^{\frac{p-1}{2}} = -1$, and also Lemma \ref{legendre} implies $(\frac{-1}{p}) = -1$.  
So by the formula above 
we have $(\frac{-3}{p}) = 1$, showing that $-3$ is a quadratic residue when $p\equiv 7($mod $12)$. 

\smallskip

\noindent \underline{\bf{case 4}}: $p\equiv 3($mod $4)$ and  $p\equiv 2($mod $3)$; equivalently $p\equiv 11($mod $12)$.

Since $p\equiv 2($mod $3)$ we have $(\frac{p}{3}) = -1$.  Again $p\equiv 3($mod $4)$ implies that $(-1)^{\frac{p-1}{2}} = -1$, and also that $(\frac{-1}{p}) = -1$.  So by the formula above 
we get $(\frac{-3}{p}) = -1$, showing that $-3$ is a quadratic nonresidue when $p\equiv 11($mod $12)$.

\medskip

Putting together cases 1 and 3, we see that $-3$ is a quadratic residue mod $p$ when $p\equiv 1($mod $6)$, while cases 2 and 4 show that $-3$ is a quadratic nonresidue mod $p$ when $p\equiv 5($mod $6)$, 
as required.

 \end{proof} 
 
 \begin{corollary}\label{minusthree}  Consider the prime power $q = p^{m}$, where $p > 3$ is an odd prime.  If $q\equiv 1($mod $3)$, 
 then $-3$ is a square in the finite field $GF(q)$.

 \end{corollary}
 
 \begin{proof} Since $p > 3$ is an odd prime we have either $p\equiv 1($mod $6)$ or $p\equiv 5($mod $6)$.  If $p\equiv 1($mod $6)$, then $-3$ is already a 
 square in the prime subfield $GF(p)\subseteq GF(q)$ by Theorem \ref{reciprocity}, so $-3$ is a square in $GF(q)$, as required.  
 
 So suppose $p\equiv 5($mod $6)$.  Consider the quadratic extension $GF(p)(\sqrt{-3})$ of $GF(p)$ obtained by 
 adjoining to $GF(p)$ a root of the irreducible (by Theorem \ref{reciprocity}) polynomial $x^{2} + 3$ over $GF(p)$.  Then $GF(p)(\sqrt{-3}) \cong GF(p^{2})$, and $-3$ is a square in $GF(p^{2})$.
 
Since $q\equiv 1($mod $3)$, then since $p\equiv 5($mod $6)$ we have $p\equiv 2($mod $3)$, so it follows that $m$ must be even.  
We recall the basic fact from finite fields that $GF(p^{r})\subseteq GF(p^{s})$ if and only if $r\vert s$. 
It follows that $GF(p^{2})\subseteq GF(q)$.  Thus since $-3$ is a square in $GF(p^{2})$, then $-3$ is a square in $GF(q)$.              \end{proof}

 \begin{corollary} \label{quadroots} Let $q = p^{m}$ be a prime power, $q\equiv 1($mod $3)$. 
 
 \smallskip
 
 \noindent \textbf{a)}The equation $x^{2} + x + 1 = 0$ has two distinct solutions in $GF(q)$.  If $x_{1}$ is such a root, 
 then $\frac{1}{x_{1}}$ is the other distinct root.
 
 \smallskip
 
 \noindent \textbf{b)}For $q$ odd and distinct $i,j\in GF(q)$, the equation $x^{2} - (i+j)x + ij + (i-j)^{2} = 0$ has 
  two distinct roots in $GF(q)$.
 
 \end{corollary}

 \begin{proof}  Consider a), and suppose first that $p$ is odd.  Since the characteristic of the field is odd, we may find the solutions by 
 the standard quadratic formula.  We obtain the solutions $x = \frac{1}{2}[ -1 + \sqrt{-3}\,], \frac{1}{2}[ -1 - \sqrt{-3}\, ]$, where we have used 
 the existence of $\sqrt{-3}$ in $GF(q)$ by Corollary \ref{minusthree}.  These solutions are distinct since $p$ is odd. 
 
 Now suppose $p=2$.  Recall the trace function $Tr_{GF(q)/ GF(2)}(x) = \sum_{i=0}^{m-1} x^{2^{i}}$, defined for any $x\in GF(q)$, which we abbreviate by $Tr(x)$.  
 It can be shown (see \cite{Pom}) that the 
 quadratic equation $ax^{2} + bx + c = 0$, with $a,b,c\in GF(2^{m})$, $a\ne 0$, has two distinct solutions in $GF(2^{m})$ if and only if $b\ne 0$ and $Tr(\frac{ac}{b^{2}}) = 0.$  
 In our case we have $a = b = c = 1$, so $\frac{ac}{b^{2}} = 1$.  Since $p = 2$ and $q\equiv 1($mod $3)$, $m$ must be even.  Thus there are an even number of terms in the sum defining 
 $Tr(x)$, each of them equal to 1.  So since the characteristic is $2$, we get $Tr(\frac{ac}{b^{2}}) = 0$ in our case.  
 It follows that $x^{2} + x + 1 = 0$ has two distinct solutions when $p = 2$, as required.
 
 Observe that if $x_{1}$ is a root of of $x^{2} + x + 1 = 0$, then by direct substitution so is $\frac{1}{x_{1}}$.  To show that $x_{1}$ and $\frac{1}{x_{1}}$ are distinct, assume not.  
 Then $x_{1} = 1$ or $-1$.  If $q$ is even, then  $x_{1}^{2} + x_{1} + 1 = 0$ implies that $1 = 0$ since the characteristic of the field is $2$, a contradiction.  Assume $q$ 
 is odd.  Then if $x_{1} = 1$ we get $1+1+1 = 0$, implying $q\equiv 0($mod $3)$, a contradiction.  If $x_{1} = -1$, then we get $1 = 0$, contradiction.  Thus $x_{1}$ and and $\frac{1}{x_{1}}$ are distinct.

 Next consider b).  Applying the quadratic formula in this field of odd characteristic, we get the two solutions  
$x = \frac{1}{2}[\,i + j \pm \sqrt{(i+j)^{2} - 4(ij + (j-i)^{2})}\,] = \frac{1}{2}[\,i + j \pm \sqrt{-3(i^{2} + j^{2}) + 6ij}\,] = \frac{1}{2}[\,i + j \pm \sqrt{-3(i - j)^{2}}\,] = \frac{1}{2}[\,i + j \pm \sqrt{-3}(i - j)\,].$  
Now since $-3$ is a square in $GF(q)$ for $q\equiv 1($mod $3)$ by Corollary \ref{minusthree}, it follows that the two solutions for $x$ can be written as  
$x_{1} =  \frac{1}{2}[i(1+\sqrt{-3}) + j(1-\sqrt{-3})]$, and $x_{2} =  \frac{1}{2}[i(1-\sqrt{-3}) + j(1+\sqrt{-3})].$ Also these two solutions are distinct since $i\ne j$ and $q$ is odd.  
\end{proof}

\end{document}